\documentclass[a4paper]{mcom-l} 

\usepackage[dvipsnames]{xcolor}
\usepackage{amsmath}
\usepackage{amssymb}
\usepackage{verbatim}
\usepackage{amsthm}
\usepackage{mathrsfs}
\usepackage{bm}
\usepackage{graphicx}
\usepackage{mathtools}
\usepackage[colorlinks,pdfdisplaydoctitle,linkcolor=blue,citecolor=blue]{hyperref}
\usepackage{caption}
\usepackage{subcaption}
\usepackage{url}
\usepackage{todonotes}
\usepackage{algorithm}
\usepackage[noend]{algpseudocode}
\usepackage{enumerate}
\usepackage{cleveref}

\newcommand{\calG}{\mathcal{G}}

\newcommand{\calH}{\mathcal{H}}

\newcommand{\bfI}{\mathbf{I}}
\newcommand{\Hbar}{\overline{\calH}}

\newcommand{\reals}{\mathbb{R}}
\newcommand{\tr}{\text{tr}}

\theoremstyle{plain}
\newtheorem{thm}{Theorem}[section]

\newtheorem{lem}[thm]{Lemma}

\newtheorem{prop}[thm]{Proposition}
\newtheorem{aspt}[thm]{Assumption}

\newtheorem{rem}[thm]{Remark}

\theoremstyle{definition}

\theoremstyle{remark}

\numberwithin{equation}{section}

\graphicspath{{./images/}}
\begin{document}

\title[Convergence acceleration of EKI in nonlinear settings]{Convergence acceleration of ensemble Kalman inversion in nonlinear settings}
\author[N. K. Chada] {Neil K. Chada}
\address{Applied Mathematics and Computational Science Program, King Abdullah University of Science and Technology, Thuwal, 23955, KSA}
\email{neilchada123@gmail.com}

\author[X. T. Tong] {Xin T. Tong}
\address{Department of Mathematics, National University of Singapore, 119077, Singapore}
\email{mattxin@nus.edu.sg}


\begin{abstract}
Many data-science problems can be formulated as an inverse problem, where the parameters are estimated by minimizing a proper loss function. 
When complicated black-box models are involved, derivative-free optimization tools are often needed.
 The ensemble Kalman filter (EnKF) is a particle-based derivative-free Bayesian algorithm originally designed for data assimilation.
 Recently, it has been applied to inverse problems for computational efficiency. 
 The resulting algorithm, known as ensemble Kalman inversion (EKI), 
 involves running an ensemble of particles with EnKF update rules so they can converge to a minimizer. 
In  this article, we investigate EKI convergence in general nonlinear settings. 
To improve convergence speed and stability, we consider applying EKI with non-constant step-sizes and covariance inflation. 
We prove that EKI can hit critical points with finite steps in non-convex settings. 
We further prove that EKI converges to the global minimizer polynomially fast if the loss function is strongly convex.
 We verify the analysis presented with numerical experiments on two inverse problems.
 \end{abstract}
 \maketitle


\section{Introduction}
A crucial task of  data science is to build mathematical models that can explain existing data, and use it to infer structural information. A general way to formulate this  mathematically is through
\begin{equation}
\label{eqn:IP}
y=\calG(u)+\eta. 
\end{equation}
From \eqref{eqn:IP}, $u\in \reals^{d_u}$ that stands for parameters of interest we try to infer from some associated noisy data $y\in \reals^{d_y}$, where $\calG$ describes the physical laws that relate $u$ and $y$. Finally $\eta$ models uncontrollable noises in the data collection process, which we assume here is {an independent} Gaussian noise, i.e. $\eta \sim \mathcal{N}(0, \Gamma )$. 
{This choice is for the simplicity of exposition. More generally, $\eta$ can be set as a multiplicative noise \cite{MMD19}, or its distribution can be non-Gaussian. The setting of $\eta$ has a crucial impact on the solution of the inverse solver due to the discrepancy principle \cite{EHNR96,MH97}.  }
%
%

Typical example of \eqref{eqn:IP} includes the subsurface flow problem, where $u$ stands for the underground geological structure, $y$ is the pressure reading at different locations, and $\calG$ involves solving  a partial differential equation (PDE) called ``Darcy's law".  In most applications, the forward problem, that is finding $\mathcal{G}(u)$ or $y$ with a given $u$, is relatively straightforward. But the associated inverse problem, that is finding $u$ with a given $y$, can be difficult. 

To solve an inverse problem \cite{MH97} given as \eqref{eqn:IP}, one standard approach is finding $u$ such that $\mathcal{G}(u)$ is closest to $y$. Mathematically, this is equivalent to minimizing the data-misfit function
\begin{equation}
\label{eq:min}
\ell_o(u)=\|\calG(u)-y\|^2_\Gamma.
\end{equation}
This is commonly referred to as the {\textit{variational approach} \cite{EHNR96,AT87}.}
Here and what follows, we use $\|v\|^2_\Gamma=v^T\Gamma^{-1}v$ to denote the Mahalanobis norm of $v$ with weight matrix $\Gamma$. 
Yet, this approach often leads to unphysical solutions that overfit the data, or there can be non-unique solutions. These issues can often be alleviated by incorporating  physical information through regularization \cite{BB18,MH97}. One popular choice is Tikhonov regularization which introduces a preference matrix $\Sigma$ and weight parameter ${\lambda}$, where now we consider minimizing the loss function
\begin{equation}
\label{eq:prior}
\ell(u)=\|\calG(u)-y\|_\Gamma^2+ {\lambda}\|u\|_\Sigma^2. 
\end{equation}
In PDE applications, the matrix $\Sigma$ is often chosen as certain Laplacian operators to enforce smoothness on $u$, and ${\lambda}$ is set as a tuning parameter. Another way of solving the inverse problem \eqref{eqn:IP} which {does not require} computing the minimizer of \eqref{eq:min} is the Bayesian approach \cite{AMS10}. This approach characterizes the solution as the conditional distribution of $u$ given $y$, i.e. $p(u|y)$, known as the posterior distribution. The physical information can be incorporated by assuming a {prior distribution $p_0(u)$ for $u$, such as $\mathcal{N}(0,{\lambda}^{-1}\Sigma)$}. {Then the posterior distribution of the inverse solution can be expressed as}
\begin{align}
\label{eq:post}
\begin{split}
{
p_1(u):=p(u|y)} &{=\frac1{\sqrt{\text{det}(2\pi\Gamma)}}\exp\left(-\frac12\|\calG(u)-y\|_\Gamma^2
\right)p_0(u)} \\ 
&{\propto \exp\left(-\frac12 \ell_0(u)\right)p_0(u).}
\end{split}
\end{align}
The variational and Bayesian approaches are closely related. In particular the minimizer of \eqref{eq:prior} is also known as the  maximum a posteriori estimator in Bayesian statistics. 

In order to minimize \eqref{eq:min} or \eqref{eq:prior}, classical optimization methods, such as {gradient descent, require first-order gradient information} of $\mathcal{G}$. However, this information can be computationally expensive if $u$ is high dimensional or if the model is complex. This is often the case for many modern-day applications. For example, in numerical weather prediction (NWP), $u$ represents the atmospheric and oceanic state on earth. Its dimension can exceed $10^8$ and $\calG$ describes the evolution of a fluid equation of multiple scales. In situations such as this, one would rather treat $\calG$ as a black-box, and apply optimization methods that are derivative-free \cite{RS13}.

The ensemble Kalman filter (EnKF) \cite{GE09,GE03} is a derivative-free algorithm designed for data assimilation problems \cite{LSZ15,SR19}, {which can be formulated as the inference of a dynamical system $u_n$ using sequential observations $y_n$ in 
\begin{equation}
\label{eqn:DA}
u_{n+1}=\mathcal{A}(u_n),\quad y_n=\calG(u_n)+\eta_n,
\end{equation}
where $\mathcal{A}$ describes the update rule of certain dynamical systems. }\eqref{eqn:DA}  can be interpreted as an inverse problem where a sequence of interrelated parameters are to be recovered. EnKF was originally derived as a Monte Carlo approximation of the Kalman filter \cite{REK60}, which is the Bayesian solution to \eqref{eqn:DA} assuming $\mathcal{A}$ and $\calG$ are linear. 
Because its formulation is derivative-free, the EnKF can be executed efficiently, and hence has been widely applied for NWP problems \cite{CR13,HM01}. The application of the EnKF to the setting of  inverse problems goes back to  \cite{ORL08}. {In short, it interprets the inverse problem \eqref{eqn:IP} as a special form of \eqref{eqn:DA} with $\mathcal{A}$ being the identity map $\mathcal{A}(u)=u$, and $y_n\equiv y$.} Since then, a wide development of work has been done on applying ensemble Kalman methods for inverse problems arising in PDEs.This was initiated by the works of Iglesias  \cite{MAI16,ILS13} and has lead to numerous further directions \cite{NKC17,CIRS17,SS17}. 
We will refer the application of EnKF to inverse problems as ensemble Kalman inversion (EKI).

As a short description, EKI draws an initial ensemble from the prior distribution or a smoothed version of it, and repeatedly applies the EnKF to update the ensemble so it fits the data better. 
In the linear setting, the continuous time approximation of EKI eventually converges to the minimizer of the data-misfit function  \eqref{eq:min} \cite{SS17}. 
{However, this minimizer in general does not contain prior information and may overfit the data.}
To avoid this issue, one approach was devised by  \cite{MAI16} which incorporated iterative {{Levenberg--Marquardt} regularization, taking motivation from an earlier work \cite{MH97}. }
Recently, { another more direct approach  is found by introducing an artificial observation in the EnKF step}, so the ensemble converges instead to the minimizer of \eqref{eq:prior}. This formulation is known as  Tikhonov ensemble Kalman inversion \cite{CST18}. Our investigation will mostly focus on it, {and for simplicity of notation, we will refer it as EKI in the discussion below. }

Despite the empirical success of EKI in the references aforementioned, its behaviour as an optimizer for \eqref{eq:prior} is not well understood. Convergence results of EKI are available only for linear observations and the continuous-time limit of EKI iterates \cite{NKC17,KLS14,KM15,GMT11,MCB11,SS17}. 
However, EKI algorithms in practice have to run at discrete time, and the observations are rarely linear. 
Moreover, recent machine learning  research has shown that using non-constant step-size or learning rate can significantly improve the optimization algorithm results \cite{DHS11,YN04,PMV01,PJ92}. For EKI, using a non-constant step-size is related to incorporating covariance inflation \cite{SS17}, 
which are important tuning techniques for improving both the accuracy and stability in NWP \cite{JLA07,JLA09,TMK16, MT18, XTT18}. 
Yet, these important features and connections can not be revealed if one investigates only the continuous-time limit. 

This paper intends to fill these gaps by investigating EKI as a derivative-free optimization tool. Our contributions are highlighted through the following:
\begin{itemize}
\item We develop a new version of the Tikhonov EKI algorithm, where non-constant step-sizes and covariance inflation are applied. 
These modifications are essential to the algorithm performance both in theory and numerical tests.
\item We compare the long time behavior of EKI with the Gauss--Newton method in a general nonlinear setting. Such comparison leads to an intuitive explanation why EKI can be used for optimization.
\item Assuming a general nonlinear map $\calG$, we show that EKI can converge to approximate critical points with finitely many iterations. If  in addition the regularized loss function \eqref{eq:prior} is strongly convex, we show that EKI converges to the global minimum at a polynomial speed. 
\item Based on our convergence analysis, we provide guidelines on how to choose the step-size and covariance inflation in EKI. We implement the EKI on the Lorenz 96 model {in 1D}, and a nonlinear elliptic partial differential equation in {1D and 2D}. {Our new implementation of EKI will be compared to the standard vanilla Tikhonov EKI. We also compare and contrast EKI with the Gauss--Newton method as much of our derived theory is based on their comparison. }
\end{itemize}

\subsection{Notation and organization}

The structure of this article is as follows. In Section \ref{sec:ESRF} we provide an overview of the preliminary material required, reviewing the EnKF, while introducing our formulation of the inverse problem loss function \eqref{eq:prior}. This leads to Section \ref{sec:main_res} where we state the main results while introducing the assumptions on the optimization and convergence analysis. Numerical verification of the results are shown in Section \ref{sec:numerical}, while finally we conclude our findings and discuss potential areas of future work in Section \ref{sec:conc}. The appendix will contain the majority of proofs from Section \ref{sec:main_res}.

{
Throughout the article we use $\| \cdot \|$ and  $\langle \cdot , \cdot \rangle$ to denote $l_2$ norm and its corresponding inner product.  For any arbitrary function, we will further denote its Jacobian and Hessian matrix as $\nabla$ and $\nabla^2$. Given a matrix $A \in \mathbb{R}^{m \times n}$ the $l_2$-operator norm is defined as 
$\| A \| = \sup_{v\in \mathbb{R}^n,\|v\|=1} \|Av\|$. Given two symmetric matrices $A$ and $B$, we use $A \succeq B$ to indicate the matrix $A-B$ is positive semidefinite. Given a covariance matrix $\Gamma$ the Mahalanobis norm  is defined by $\|v\|^2_\Gamma=v^T \Gamma^{-1}v$. }

\section{Tikhonov ensemble Kalman inversion}
\label{sec:ESRF}

In this section we provide the key steps for deriving the Tikhonov EKI algorithm. 
We initiate with  an overview of optimization with iterative Bayesian approaches, while discussing how to implement EKI with non-constant step-sizes and covariance inflations. For notation simplicity, we assume ${\lambda}=1$ in our discussion. This does not sacrifice any generality, since we can always replace $\Sigma$ with ${\lambda}^{-1}\Sigma$ otherwise.

\subsection{Optimization by iterative Bayesian sampling}
\label{sec:optBay}
{
The idea of using sampling methods for optimization can be traced back to simulated annealing \cite{BT93,KGV83}. To implement it, we try to sample the distribution $\pi_N(u)\propto \exp\left(- \frac N{2}\ell_o(u)\right)$ with increasingly large $N$. Then most of the samples will concentrate around the minimizer of $\ell_o$. In practice, $\pi_N(u)$ can be difficult to sample directly when $N$ is large, because typical Markov Chain Monte Carlo algorithms may be trapped at one mode of $\ell_o$ and fail explore other modes. One common strategy to resolve this issue is to use sequential Monte Carlo or tempering techniques \cite{BCJ14,DDJ06, KBJ14}, so the algorithm can sample a sequence of distributions $\pi_i$, where the update from $\pi_i$ to $\pi_{i+1}$ is simple to achieve. In below, we explain how to do this by iterating Bayesian algorithms.} 


{Bayesian algorithms are designed to obtain posterior distribution samples when the prior and observation data are given. In particular, suppose $u$ follows a prior distribution $p_0$, while the observation $y$ is modeled by \eqref{eqn:IP}, Bayesian algorithms intend to sample the posterior distribution}
\[
p_1(u)=\frac1{\sqrt{\text{det}(2\pi\Gamma)}}\exp\left(-\frac12\|\calG(u)-y\|_\Gamma^2\right)p_0(u)\propto \exp\left(-\frac12 \ell_o(u)\right)p_0(u). 
\]
Then suppose we use $p_1$ as the prior {formally}, and the  data $y$ as the observation again, the next ``posterior" distribution is given by
\[
p_2(u)=\frac1{\sqrt{\text{det}(2\pi\Gamma)}}\exp\left(-\frac12\|\calG(u)-y\|_\Gamma^2\right)p_1(u)\propto \exp\left(- \ell_o(u)\right)p_0(u). 
\]
If we iterate this procedure $N$ times, the resulting posterior is given by 
\begin{equation*}
\label{eqn:pN}
p_N(u)\propto  \exp\left(- \frac N{2}\ell_o(u)\right)p_0(u).
\end{equation*}
When $N$ is large enough, most of the probability mass of $p_N$ will concentrate on the minimum of $\ell_o$. {This is similar to the idea of simulated annealing.}

{It is worthwhile emphasizing the Bayesian algorithm is used here only to solve an optimization problem, as we are not trying to solve a Bayesian problem. Otherwise, the same data $y$ should not be used iteratively. If we are interested in  sampling, for example $p_1$, with EKI techniques, the distribution sequence needs to be adjusted by tempering techniques,  as discussed in \cite{SS17}.
}


\subsection{Regularized observation}
\label{sec:regular}
As discussed in the introduction, the minimizer of the data-misfit function $\ell_o$ may be a nonphysical solution which overfits the data. It is often more desirable to minimize the regularized loss function $\ell$ \eqref{eq:prior} instead. {Therefore} we need to include the regularization term into the observation model. This is achieved by concatenating the real observation $\calG(u)$ with a direct artificial  observation with observation noise $N(0,\Sigma)$. {Our setup will follow almost identically to that of \cite{CST18}}.  We begin by extending \eqref{eqn:IP} to the equations
\begin{subequations}
\label{eq:invT}
\begin{align}
\label{eq:inv_re1}
y&=\calG(u) + \eta, \\ 
\label{eq:inv_re2}
u&= \zeta, 
\end{align}
\end{subequations}
where $\eta, \zeta$ are independent random variables
distributed as $\eta \sim N(0,\Gamma)$ and $\zeta \sim N(0,\Sigma).$
Define  variables $z,\xi$ and  mapping $\calH: \reals^{d_u} \mapsto
\reals^{d_y} \times \reals^{d_u}$ as follows,
\[
z=\begin{bmatrix}
y\\
0
\end{bmatrix},\quad
\calH(u)=\begin{bmatrix}
\calG(u)\\
u
\end{bmatrix},
\quad
\xi=\begin{bmatrix}
\eta\\
\zeta
\end{bmatrix}.
\]
Then note that
\[
\xi \sim N(0,\Gamma_+), \quad
\Gamma_{+} =
\begin{bmatrix}\
\Gamma & 0\\
0 &\Sigma
\end{bmatrix}.
\]
We can express our modified inverse problem as
\begin{equation}
\label{eqn:IPnew}
z = \calH(u) + \xi.
\end{equation}
Under this transformation, the regularized loss function $\ell$ in \eqref{eq:prior} can be express as the data-misfit function of \eqref{eqn:IPnew}:
\begin{equation*}
\ell(u)=\| \Gamma^{-1/2}(\calG(u) -y)\|^2+ \|\Sigma^{-1/2}u\|^2=\| (\calH(u)-z)\|^2_{\Gamma_+}.  
\end{equation*}
\subsection{Kalman filter and ensemble formulation}
\label{ssec:enkf}
When {the distribution $p_n$} follows Gaussian $\mathcal{N}(b_n, \Sigma_n)$ and $\calH(u)=Hu$ is linear, the Kalman filter \cite{REK60, TJS14}
 provides explicit formulas for the posterior distribution $p_{n+1}$ in \eqref{eq:post}. In particular, $p_{n+1}$ is given by $\mathcal{N}(b_{n+1}, \Sigma_{n+1})$, where 
\begin{align}
\label{eqn:Kalman}
\begin{split}
{\Sigma_{n+1}}&{=\Sigma_n-\Sigma_n H^T ( \Gamma_++ H\Sigma_n H^T )^{-1}H\Sigma_n,} \\  
{b_{n+1}}&{=\Sigma_{n+1} \Sigma_n^{-1}b_n +\Sigma_{n+1}H^T\Gamma_+^{-1}z.} 
\end{split}
\end{align}
Iterating the same formula, one can find the sequential distributions $p_n$ are all Gaussian, which implies the mean and covariance all have explicit forms. 

In {practice}, applying the Kalman filter can be difficult, as $\calH$ may be nonlinear, and inverting the associated matrices can be expensive if the underlying dimension $d_u$ is large. The EnKF algorithm is designed to overcome these two issues. It uses a group of particles $\{u^{(i)}_n\}_{i=1}^{K}$ to represent the Gaussian distribution $p_N$, where the covariance matrices in \eqref{eqn:Kalman} can be approximated by their emperical versions. 

In particular, we define
\begin{equation}
\label{eq:samp_mean}
m_n=\frac1K\sum_{i=1}^K u^{(i)}_n, \quad \Hbar_n=\frac{1}{K}\sum_{i=1}^K \calH(u_n^{(i)}).
\end{equation}
and the sample covariances
\begin{subequations}
\label{eq:samp_cov}
\begin{alignat}{4}
C^{uu}_n&=\frac1K\sum_{i=1}^K (u^{(i)}_n-m_n) \otimes(u^{(i)}_n-m_n), \\ C^{pp}_n&=\frac1K\sum_{i=1}^K (\calH(u^{(i)}_n)-\Hbar_n) \otimes(\calH(u^{(i)}_n)-\Hbar_n),
\\
C^{pu}_{n}&= \frac1K\sum_{i=1}^K (\calH(u^{(i)}_n)-\Hbar_n) \otimes (u^{(i)}_n-m_n),\quad C^{up}_{n}= (C^{pu}_n)^T.
\end{alignat}
\end{subequations}
Suppose $u^{(i)}_n$ are i.i.d. samples from $\mathcal{N}(b_n,\Sigma_n)$ and $\calH(u)=Hu$, it is evident that $m_n$, $C^{uu}_n$, $C^{pp}_n$, and $C^{up}_n$ are 
 approximations of $b_n, \Sigma_n, H\Sigma_n H^T,$ and $\Sigma_nH^T$. By inserting these approximations in \eqref{eqn:Kalman}, {we attain} $p_{n+1}$. {We refer the reader to various references regarding analysis of the EnKF in the large ensemble limit \cite{LS19a,GMT11}}. There are in general two ways to update the particles such that their mean and covariance satisfy \eqref{eqn:Kalman}. The first way is directly updating the particles by 
\begin{align}
\label{eqn:particleupdate}
u^{(i)}_{n+1}&=u^{(i)}_n+C^{up}_n(C^{pp}_n +\Gamma_+ )^{-1}(z+\xi_{n}^{(i)} -\calH(u_n^{(i)})),
\end{align}
where $\xi_n^{(i)}$ are i.i.d. samples from $\mathcal{N}(0,\Gamma_+)$. With these artificial noises, one can show that on average, the mean and covariance of $\{u^{(i)}_{n+1}\}_{i=1}^K$ are approximately  $(b_{n+1}, \Sigma_{n+1})$ in \eqref{eqn:Kalman}. 
{If we consider the application of \eqref{eqn:particleupdate} to inverse problems, the resulting methodology is known as original ensemble Kalman inversion in \cite{ILS13,ORL08}. As can be told from \eqref{eqn:particleupdate}, this algorithm is derivative-free.}

 On the other hand, adding artificial noises $\xi_n^{(i)}$ in \eqref{eqn:particleupdate} creates fluctuation and instability.  The second way is simply finding a group of particles such that their mean and covariance match the target formulas in \eqref{eqn:Kalman}. This leads to the mean update 
\begin{equation}
\label{eqn:meanupdate}
m_{n+1}=m_n+ C^{up}_n( C^{pp}_n +\Gamma_+ )^{-1}(z -\calH(m_n)). 
\end{equation}
Then we seek a new ensemble centered at $m_{n+1}$, so that 
\begin{equation*}
C^{uu}_{n+1}=C^{uu}_n- C^{up}_n ( \Gamma_++ C^{pp}_n )^{-1}C_n^{pu}. 
\end{equation*}
This formulation is often called the ensemble square root filter (ESRF)\cite{LDN08,TAB03}. 
{There are several different ways to update the ensemble so this holds, either through factorizing the associated matrices or formulating it as an optimal transport problem \cite{LSZ15,MH12,RC15}.}
ESRF is known to perform better than the particle formulation \eqref{eqn:particleupdate}.  {The large ensemble limit of it has been analyzed in a similar context \cite{KM15,LS19a}.}

{It is worth mentioning that Kalman filter formulas \eqref{eqn:Kalman} produce the accurate distribution updates only when $\mathcal{H}$ is linear and the underlying distributions $p_n$ are Gaussian. 
Alternatively, these formulas can also be derived through a variational approach where the Gaussian assumption is not necessary
(see e.g. Section 7.2 of \cite{TJS14}).
When the observation model $\mathcal{H}$ is nonlinear, Kalman filter in general can be inaccurate. Variants such as the extended Kalman filter 
\cite{RGY99} and unscented Kalman filter \cite{KU97,WM00} are designed to handle the nonlinearity.
 In particular, EnKF, as a simple particle implementation of Kalman filter,}  is routinely applied for nonlinear DA problems such as numerical weather forecast, and it yields reasonably good forecast skills. The exact mathematical reason in behind is still largely unknown despite active researches in this direction \cite{KLS14, KMT16}. Our investigation of EKI in below aims to shed light on this issue, but in the context of inverse problems. {In particular, our results below hold for general nonlinear observation model $\calH$ and require no Gaussian assumptions.}

\subsection{Non-constant step-size and covariance inflation}
Recall that in the gradient descent (GD) algorithm, one generates a sequence of iterates $u_n$ to approach the minimum of $\ell(u)$. One way to update the iterate is by 
\[
u_{n+1}=u_n-h\nabla \ell(u_n),
\]
where $h>0$ is often called the step-size. One can interpret GD as implementing Euler's method for the gradient flow of an ordinary differential equation.  {In the optimization literature, it is shown that using non-constant sequence $h_n$ may improve algorithm performance \cite{YN04}. In particular, decreasing step-sizes allow the algorithm to take larger steps at earlier iterations and explore more regions, while converging to a solution faster with smaller steps in later iterations. Moreover, one can implement  Armijo rule or Wolfe condition to further improve convergence \cite{NW06}.  }

One of the findings in this paper and some earlier works is that the EKI in the long run similar to GD \cite{HLR18,KS18}. One would naturally conjecture that implementing non-constant step-size may lead to improved optimization performance. {In \cite{SS17}, it is shown that to implement EKI with a constant small step-size $h$, we simply replace $\Gamma_+$ with $h^{-1}\Gamma_+$. One can also reach such procedure by considering applying tempering techniques from sequential Monte Carlo when sampling \eqref{eq:post} \cite{DDJ06,MPT18}. } Here we implement the same idea, except that we explore the possibility of using the step-size $h_n=h_0 n^{\beta}$ in place of $h$. When $\beta=0$, this is the same taking a constant step-size. As a result, the mean update formula is given by 
\begin{equation}
\label{eqn:EKImean}
m_{n+1}=m_n+  C^{up}_n( C^{pp}_n +h_n^{-1}\Gamma_+ )^{-1}(z -\calH(m_n)). 
\end{equation}
{One interesting fact we found is that } the step-size for EKI does not need to decrease. This is because the movement made by \eqref{eqn:EKImean} is closer {to a}  Gauss--Newton type of algorithm, instead of a GD type of algorithm. Yet, the step-size parameter does control the final convergence speed. This will be clearer when we have more analysis results. Specifically Remark \ref{rem:step-size} will provide further details. {In the continuous-time setting of EKI, adaptive time-steppings have also been introduced in \cite{CST18,KS18} which work well, however nothing in terms of the discrete-time setting.}

%
Aside from implementing the non-constant step-size, we will also apply additive covariance inflation \cite{JLA07,JLA09,TMK16} for the update formula. The resulting covariance update is given by
\begin{equation}
\label{eqn:EKIcov}
C^{uu}_{n+1}=C^{uu}_n- C^{up}_n (C^{pp}_n +h^{-1}_n \Gamma_+)^{-1}C_n^{pu}+\alpha^2_n \Sigma,
\end{equation}
where $\alpha_n$ is a sequence of positive parameters to be specified. In the literature of EnKF, covariance inflation is commonly applied for improved algorithm stability. It is {similar} to adding a stochastic noise in the particle formulation \eqref{eqn:particleupdate}. {In particular, one can verify the sample covariance of 
\[
u^{(i)}_{n+1}=u^{(i)}_n+C^{up}_n(C^{pp}_n +h^{-1}_n\Gamma_+ )^{-1}(z+\xi_{n}^{(i)} -\calH(u_n^{(i)}))+\alpha_n\zeta^{(i)}_n
\]
with independent  $\zeta^{(i)}_n\sim \mathcal{N}(0,\Sigma)$ will follow \eqref{eqn:EKIcov} when $N$ goes to infinity.}
Such operation is also known as jittering or rejuvenation in data assimilation \cite{GSS93}, and it in general makes the associated Kalman filter system controllable. 

{It is worth mentioning that there might be several different  ways to implement the ERSF ensemble update so that \eqref{eqn:EKIcov} holds. {We provide one possible implementation as Algorithm \ref{alg:1}  in the appendix. }Our discussion below does not rely on the particular choice of numerical method, but only the relation \eqref{eqn:EKIcov} itself. Also, since the rank of $N$-sample covariance matrix is at most $N-1$, and \eqref{eqn:EKIcov} is in general full rank, so to implement this version of EKI, the ensemble size needs to be larger than the dimension of $u$. When the dimension of $u$ is large, this can be computationally expensive. But it can be partially resolved by selecting a proper subspace to implement EKI, which is discussed in \cite{CIRS17}}.

\section{Main results}
\label{sec:main_res}
In this section we state our analysis results regarding the convergence  of Tikhonov EKI. We first aim to understand the behaviour of the ensemble, more specifically the ensemble covariance $C^{uu}_n$. We then compare the EKI update and  Gauss--Newton (GN) update, where we show that their difference converges to zero.  Finally we state results regarding both convergence towards local and global minimizers. The proofs of these results will be omitted from this section and are provided in the appendices.
\subsection{Ensemble covariance collapse}
The first step of our analysis involves understanding the ensemble configuration of EKI. 
For simplicity, we impose the following regularity assumption for the map $\calH$.
\begin{aspt}
\label{aspt:obsLip}
$\calH$ has bounded first and second order derivatives. So there are constants $M_1$ and $M_2$ such that for all $z,z'$ and $v$ in $\reals^{d_u+d_y}$
\[
\|\nabla \calH(z)\|\leq M_1,\quad \|\calH(z')-\calH(z)\|\leq M_1\|z'-z\|,\quad v^T\nabla^2 \calH(z) v \leq M_2 \|v\|^2. 
\]
\end{aspt}
\begin{thm}
\label{thm:precision}
Under Assumption \ref{aspt:obsLip}, suppose the EKI algorithm \eqref{eqn:EKImean}-\eqref{eqn:EKIcov} is implemented with $h_n=h_0 n^\beta$ and $\alpha_n^2=\alpha_0^2h_0^{-1} n^{2\gamma-\beta-2}$, where the parameters satisfy  
\[
\gamma-1\leq\beta\leq \gamma.
\]
 Then the sample covariance $C^{uu}_n$ is bounded from above and below for all $n\geq 1$,
\[
\kappa_m n^{\gamma-\beta-1} \Sigma \preceq C^{uu}_n \preceq \kappa_M n^{\gamma-\beta-1} \Sigma\quad \text{with  constants  }\kappa_m,\kappa_M>0.
\]
%
\end{thm}

In the view of classical linear Kalman filter theory \cite{REK60}, this result indicates that the system is observable and controllable.
{It is worth mentioning that the lower bound of sample covariance is non-trivial. In the vanilla  EKI, such lower bound cannot be derived \cite{CST18,SS17}. And if the ensemble happens to collapse onto the same point, the vanilla EKI will stagnate at that point. In contrast, the balanced step-size and covariance inflation we implement here keep the ensemble from collapsing too fast and premature stagnation.}

\subsection{Connection with Gauss--Newton}
\label{sec:GN}

{The Gauss--Newton (GN) method  is a popular choice as an optimizer for non-linear least squares problems.
Its application in solving inverse problems has been well-documented and studied \cite{BHM09, EHNR96}. {In particular, when applying it to minimize a loss function of form $l(v)=\|Q(v)-z\|^2$, it generates iterates $v_n$ by running 
\begin{equation}
\label{eqn:GNupdate}
v_{n+1}=v_n-(\nabla Q(v_n)^T\nabla Q(v_n) )^{-1} \nabla Q(v_n)^T(Q(v_n)-z).
\end{equation}}}
{In the Kalman filter literature, it is a known fact that the Kalman filter and the extended Kalman filter are closely related to the GN method, since their mean update \eqref{eqn:Kalman} can be seen as the update rule \eqref{eqn:GNupdate}. 
We will show  in this section, that as an ensemble formulation of the Kalman filter, EKI inherits such a connection in the long run. This intuitively explains why EKI is an appropriate optimization tool. It is also worth pointing out that both Gauss--Newton and extended Kalman filter need first order gradient information to implement, so it is nontrivial that EKI can achieve similar update rule without gradient information.}



{Given that our iterated posterior} $p_n$ is assumed to be $\mathcal{N}(m_n, C^{uu}_n)$, and that $p_{n+1}$ is assumed to be Gaussian, the mean of $p_{n+1}$,  $m_{n+1}$,  should be the minimizer of $-\log p_{n+1}$, which is proportional to 
\begin{equation}
\label{eqn:elln}
\ell_{n+1}(u)= \|u-m_n\|^2_{C^{uu}_n}+\|\calH (u)-z\|^2_{\Gamma_+ h_n^{-1}}. 
\end{equation}
Note that we have replaced $\Gamma_+$ with $\Gamma_+ h_n^{-1}$ to implement our non-constant step-size.  Since $\ell_{n+1}$ is of nonlinear-least-square form, given the current mean $m_n$, the GN method indicates that $m_{n+1}$ should be $m_n+G_n$, where
\begin{align}
\notag
G_n&=[h_n^{-1}(C_n^{uu})^{-1}+ J_n^T\Gamma_+^{-1} J_n]^{-1} J^T_n \Gamma_+^{-1} (z -\calH(m_n))\\
\label{eqn:Gn}
&=C^{uu}_nJ_n^T ( J_nC^{uu}_nJ_n^T+ h^{-1}_n\Gamma_+)^{-1}(z -\calH(m_n)),\quad J_n:=\nabla \calH(m_n). 
\end{align}
Because $\|( J_nC^{uu}_nJ_n^T+ h^{-1}_n\Gamma_+)^{-1}\|\leq h_n \|\Gamma_+^{-1}\|$, by Theorem \ref{thm:precision} we can estimate 
\begin{equation}
\label{eqn:Gnsize}
\|G_n\|\leq O(\|C^{uu}_n\|h_n)=O(n^{\gamma-1}).
\end{equation}
Next, recall the mean movement from EKI \eqref{eqn:EKImean} is given by
\begin{equation}
\label{eqn:Deltan}
\Delta_n:=m_{n+1}-m_n=C^{up}_n( C^{pp}_n +h_n^{-1}\Gamma_+ )^{-1}(z -\calH(m_n)).
\end{equation}
This is different from \eqref{eqn:Gn},  however we can show their difference converges to zero. To see that, recall from Theorem \ref{thm:precision}, we find the ensemble covariance $C_n^{uu}$ decreases to zero in a  well controlled manner. In particular, the particles $u^{(i)}_n$ are very close to the mean $m_n$ when $n$ is large. This indicates the ensemble spread $\Delta u^{(i)}_n=u^{(i)}_n-m_n$ is very small. We can apply a first order approximation:
\[
\calH(u^{(i)}_n)\approx \calH(m_n)+J_n \Delta u^{(i)}_n.
\]
With this approximation, we find that 
\[
C^{up}_n\approx C^{uu}_n J_n^T,\quad  C^{pp}_n\approx J_n C^{uu}_n J_n^T.
\]
Applying these approximations to \eqref{eqn:Deltan}, we recover \eqref{eqn:Gn}. More specifically, the difference between EKI mean update and Gauss--Newton update is bounded, as discussed in the following proposition.
 \begin{prop}
 \label{prop:GN}
Under the setting of Theorem \ref{thm:precision}, there is a constant $M_3$, such that for sufficiently large $n$ the following bound holds:
\[
\|G_n-\Delta_n\|\leq M_3h_n K \|C_n^{uu}\|^{\frac{3}{2}}  \|z-\calH(m_n)\|.
\]
Given the estimates in Theorem \ref{thm:precision}, the upper bound above is of order $O(n^{\frac32 \gamma-\frac32-\frac12\beta})$, which will converge to zero with large $n$.  
\end{prop}
{\begin{rem} 
\label{rem:step-size}
Recall that in \eqref{eqn:Gnsize} we show the mean movement made by EKI is of order $n^{\gamma-1}$. So the difference between EKI and Gauss--Newton is of a lower order. Also note that the step-size parameter $\beta$  actually does not control EKI mean movement. Instead, it controls the speed of ensemble collapse as in Theorem \ref{thm:precision}, and consequentially the accuracy of EKI in approximating Gauss--Newton. 
\end{rem}}

\subsection{Iterative descent made by EKI}
While Proposition \ref{prop:GN} explains how the EKI iterates optimize a sequence of loss functions \eqref{eqn:elln}, it is unclear how the  regularized loss function \eqref{eq:prior} is optimized in the process. $\ell$ is not necessarily the limit of $\ell_n$, since $m_{n}$ is not a fixed point.  One interesting fact is that by running a Gauss-Netwon type update for $\ell_{n+1}$, the value of $\ell$ is also decreased at each step. 
To see this, we again apply the Taylor expansion 
\[
\calH(m_{n}+G_n)\approx \calH(m_n)+J_n G_n,
\] and find that
\[
 \ell (m_n+G_n)=\|\calH(m_n+G_n)-z\|_{\Gamma_+}^2\approx \ell (m_n)- 2\langle \Gamma_+^{-1}(\calH(m_n)-z), J_n G_n\rangle.
\]
The important observation here is that 
\begin{align*}
&\langle \Gamma_+^{-1}(\calH(m_n)-z), J_n G_n\rangle\\
&=(J^T_n \Gamma_+^{-1}(z -\calH(m_n)))^T[(h_n C_n^{uu})^{-1}+ J_n^T\Gamma_+^{-1} J_n]^{-1} J^T_n \Gamma_+^{-1}(z -\calH(m_n))\\
&=\|J^T_n \Gamma_+^{-1} (z -\calH(m_n))\|^2_{(h_n C_n^{uu})^{-1}+J_n^T\Gamma_+^{-1} J_n}\geq 0.
\end{align*}
It is easy to check that
\[
J^T_n \Gamma_+^{-1} (z -\calH(m_n))=-\nabla \ell (m_n).
\]
Since Proposition \ref{prop:GN} suggests that $m_{n+1}\approx m_n+G_n$, 
\[
\ell(m_{n+1})\approx \ell(m_n)-2\|\nabla\ell(m_n)\|^2_{(h_n C_n^{uu})^{-1}+J_n^T\Gamma_+^{-1} J_n}.
\]
The error of this approximation is given by the following.
\begin{prop}
\label{prop:perturb}
Under the same setting as Proposition \ref{prop:GN} we have the following estimate 
\[
\ell(m_{n+1})=\ell(m_n)-2\|\nabla \ell(m_n)\|^2_{(h_n C_n^{uu})^{-1}+J_n^T\Gamma^{-1}_+ J_n}+R_n, 
\]
where the residual is bounded by
\[
|R_n|\leq M_4 h_n \|C_n^{uu}\|^{\frac{3}{2}}\max\{\|z-\calH(m_n)\|^4,1\}.
\]
\end{prop}

\subsection{Convergence analysis}
Classical analysis  of optimization algorithms often focus on  understanding the limiting behavior of the iterations. When the underlying loss function is strongly convex, there is a unique global minimum, so it is of interest to show the algorithms can converge to this minimizer with finite steps. Under non-convex settings, the global minimum can be non-unique, and it is more practical to ask whether the algorithm can converge to a critical point of the loss function. 

The EnKF is known to have certain stability issues \cite{TMK16}, in the sense the iterates in principle may diverge to infinity. With general observation functions, EKI can have the same phenomena. But this issue can often be fixed by modifying the algorithm, if we know  proper solutions should be bounded by a known radius $M$. Such information can often to be obtained from the physical background of the inverse problem. As a consequence, it is reasonable to modify the EKI algorithm so the particles are bounded. One simple way to achieve this is by modifying the observation map outside the radius \cite{CST18}. In particular, we have the proposition.
\begin{prop}
\label{prop:ESRF}
Suppose the observation map $\calG(u)$ takes value $\mathbf{0}$ when $\|u\|\geq M+1$,   there is a threshold iteration  $n_0$, so that the EKI sequence is bounded such that
\[
\|m_n\|\leq \max\{2M+2+\|z\|,\|m_{n_0}\|\}\quad \forall n\geq n_0. 
\]
\end{prop}
The requirement that $\calG(u)$ takes value $\mathbf{0}$ when $\|u\|\geq M+1$ can be enforced for general observation function $\calG$ by multiplying it with a mollifier, for example we replace $\calG$ with
\begin{equation}
\label{eq:moll}
\widetilde{\calG}(u)= \calG(u) \exp(-C((M+1)^2-\|u\|^2)^{-1})\quad \text{where $C$ is a large constant}. 
\end{equation}
Proposition \ref{prop:ESRF} indicates it is reasonable to assume the mean sequence is bounded. Then from the descend estimate in Proposition \ref{prop:perturb}, we can show EKI will reach an  approximate critical point with finite iterations.
\begin{thm}
\label{thm:critical}
Under the setting of Theorem \ref{thm:precision}, suppose that the EKI mean sequence $m_n$ is bounded and the parameter $\gamma\in [0,1)$, then or any $\epsilon>0$,  
\[
\min_{n\leq N_\epsilon}\{\|\nabla \ell(m_n) \|\}\leq \epsilon.
\]
The threshold iteration is given by 
\[
N_\epsilon=\begin{cases} \exp(D/\epsilon^2)\quad &\text{if}\quad \gamma=0,\\
(D/\epsilon^2)^{\frac{2}{\min\{2\gamma,\beta+1-\gamma+\delta\}}} \quad &\text{if}\quad 1>\gamma>0.
\end{cases}
\]
$D$ here is a constant independent of $\epsilon$. 
\end{thm}

In addition if we assume the loss function is strongly convex, then we have the following theorem which establishes convergence to the global minimizer. The theorem also states that the convergence is attained at a polynomial rate.  

\begin{thm}
\label{thm:convex}
Under the same setting as Theorem \ref{thm:precision}, suppose in addition that  $\ell(u)$ is strongly convex, so there is a $\lambda_c>0$ such that for any vectors $x,y$
\[
\ell(x)-\ell(y)\geq \langle \nabla \ell(y), x-y\rangle+\lambda_c \|x-y\|^2. 
\]
Then there is a threshold iteration $n_0$ and constant $D$ so that the following estimates hold for any $N\geq n_0$:
\begin{enumerate}[1)]
\item If $\gamma=0$, for any $\alpha<\min\{\frac12+\frac12\beta, h_0 \kappa_m \sigma_m\lambda_c \}$, 
\[
\lambda_c \|m_N-u^*\|^2\leq \ell (m_N)-\ell (u^*)\leq \frac{D}{N^\alpha}. 
\]
Here $\kappa_m$ is given by Theorem \ref{thm:precision} and $\sigma_m$ is the minimum eigenvalue of $\Sigma$. 
\item If $1>\gamma>0$, for any $\alpha<\frac12+\frac12\beta-\frac12 \gamma$, 
\[
\lambda_c \|m_N-u^*\|^2\leq \ell(m_N)-\ell(u^*)\leq \frac{D}{N^\alpha}. 
\]
\end{enumerate}
\end{thm}

%

\section{Numerical results}
\label{sec:numerical}
{In this section we present several experiments assessing the performance of EKI. This will include monitoring the effect of the additive covariance inflation and non-constant step-sizes, whilst accounting for the computational time taken. {To highlight this we also compare EKI with the vanilla Tikhonov EKI and the GN method}. We test the inversion performance on both the Lorenz 96 model and a nonlinear partial differential equation (PDE) from the field of geosciences.   {Before describing in detail the experiments we present the different test models with their corresponding inverse problem, and discuss how we propose to compare the different methodologies introduced.} }


 \subsection{Test models}

\subsubsection{Lorenz 96 model}
The first test problem is the Lorenz 96 (L96) model \cite{ENL96}. The L96 model is a dynamical system designed to describe equatorial waves in atmospheric science. The L96 model takes the form
\begin{equation}
\label{eq:l96}
\begin{gathered}
\frac{dv_k}{dt} = v_{k-1}(v_{k+1} - v_{k-2}) - v_k + F, \quad k=1,\ldots,N, \\
v_0 = v_N, \quad v_{N+1} = v_1, \quad v_{-1} = v_{N-1}.
\end{gathered}
\end{equation}
In \eqref{eq:l96}, $v_k$ denotes the current state of the system at the $k$-th grid point. $F$ is a forcing constant with default value $8$. The dimension $N$ is often chosen as $40$, but other large numbers can also be used. {The initial condition $$v(0)=(v_1(0),\ldots,v_N(0))^T,$$ of \eqref{eq:l96} is randomly sampled from the Gaussian approximation of the equilibrium measure of L96.} {We also generate the initial ensemble for EKI from the same Gaussian distribution.}
The associated inverse problem is the recovery of $v(0)$  using 20 noisy partial measurements at time $t=0.3$ 
\begin{equation}
\label{eq:o}
y_k = v_{2k-1}(t) + \eta_k.
\end{equation}
{This inverse problem is a standard test problem for data assimilation. It is also a good testbed for high dimensional Bayesian computational methods, since the dimension $N$ can take arbitrary large values \cite{LT19,MTM19}.} {To solve the L96 model we use a fourth order Runge--Kutta method with step size $h_{\textrm{L96}} = 0.01$.}




\subsubsection{1D Darcy's law}

The second test problem will be a nonlinear PDE motivated from geosciences referred to as Darcy's law. Assume that in a domain $\mathcal{D}$ we have a source field $f \in L^{\infty}(\mathcal{D})$ and a diffusion coefficient $\kappa \in L^{\infty}(\mathcal{D})$, referred to as the permeability, then the forward problem is concerned with solving $p \in H^1_0(\mathcal{D})$, known as the pressure,  in the PDE 
\begin{align}
\label{eq:darcy}
-\nabla\cdot({\kappa}\nabla p) &= f, \quad x \in  \mathcal{D}, \\
p &= 0, \quad x \in  \partial \mathcal{D}. \nonumber
\end{align}
{We impose a Dirichlet boundary condition on the PDE} and  the inverse problem concerned with \eqref{eq:darcy} is the recovery of the permeability $\kappa$ from measurements of the pressure $p$ at 25 equidistance locations $\{x_j\}_{j=1}^{25}$. The associated inverse problem is then defined as
\begin{equation}
\label{eq:func}
 y_j =p(x_j) + \eta_j, \quad j=1,\cdots,J, 
\end{equation}
where the $\eta_j$ are Gaussian noises and assumed independent.
By defining $\mathcal{G}_j(u) = p(x_j)$, we can rewrite \eqref{eq:func} as the inverse problem 
\begin{equation*}
y = \mathcal{G}(u) + \eta, \quad \eta \sim N(0,\Gamma).
\end{equation*}
{As a starter, we will consider a 1D Darcy flow problem,  where the domain is given as $\mathcal{D}=[0,\pi]$. By testing on this toy problem, we can have a more direct observation of the recovery skill with different parameter setups, as shown later by Figure \ref{fig:PDE}.}
 The initialization of the ensemble is taken to be a mean-zero Gaussian with covariance function $C(x,x') = 5\exp\big(-\frac{|x-x'|}{20}\big)$, which is a common covariance function used in the context of uncertainty quantification \cite{LPS14}.  {To solve the Darcy's law model we use a centered finite-difference method with a mesh of $h_{\textrm{PDE}} = 1/100$.}

\subsubsection{2D Darcy's law}
Our final test problem is Darcy's law as stated through \eqref{eq:darcy}, but we consider in two dimension. As before we are interested in the recovery of the permeability $\kappa$ in the domain $\mathcal{D} = [0,1]^2$ where we specify $8 \times 8 =64$ equidistant pointwise observations. However one key difference in the 2D setting is the initialization of TEKI. We follow the setting described in \cite{CST18} and draw each initial ensemble member through the series expansion
\begin{equation}
\label{eq:KL}
u=\sum_{k \in \mathbb{Z}_+^2} \sqrt{\lambda_k}\xi_k \varphi_k(x), \quad \xi_k \sim N(0,1), \quad\rm{i.i.d.,}\
\end{equation}
\eqref{eq:KL} is known as the Karhunen-Lo\`{e}ve (KL) expansion, where $\varphi_k$ and $\lambda_k$ are the respective eigenfunctions and eigenvectors of the covariance operator $C$, defined as
$$\varphi_k(x)=\sqrt{2}\sin(\pi {\langle k, x\rangle}), \quad \lambda_k=\left(|k|^2\pi^2+\tau^2\right)^{-\nu}, \quad k \in \mathbb{Z}_+^2.$$
The corresponding KL expansion \eqref{eq:KL}  satisfies the eigenvalue problem $C\varphi_k=\lambda_k \varphi_k,$ where our covariance operator is defined as
$$C=\left(-\triangle+\tau^2\right)^{-\nu},$$
with $\triangle$ denoting the Laplacian operator,
 such that $\nu \in \mathbb{R}^{+}$ denotes the regularity and $\tau \in \mathbb{R}^{+}$ denotes the inverse lengthscale. {For the tests below, $\tau=15$. For simplicity, we initiate EKI with KL basis vectors, in other words we let}
\begin{equation}
\label{eq:kl_basis}
u^{(k)}(x)=\varphi_{k}(x), \quad  k=1,\ldots,K.
\end{equation}
This has been shown in \cite{CST18} to work well in the context of TEKI. The truth for this test model is given in Figure \ref{fig:truth2D}. We specify the observation noise covariance as in the 1D Darcy flow experiments.

\subsection{Parameter settings for the EKI methods}
\label{ssec:param_sett}

{To test our modified version of EKI \eqref{eqn:EKImean}, we monitor the effect of our implemented step-size $h_n=h_0 n^\beta$, and our additive covariance inflation} 
\begin{equation}
\label{eq:inflat}
{\alpha^2_n = \alpha^2_0h^{-1}_0n^{2\gamma-\beta-2}}.
\end{equation}
{As the form of our covariance inflation \eqref{eq:inflat} includes the step-size $h_n$, we consider different cases for our inflation based on modifying the parameters $\beta$ and $\gamma$. The parameter $\beta$ controls the step-size $h_n=h_0 n^\beta$. The other parameter $\gamma$ arises from \eqref{eq:inflat} and controls the inflation. In order for Theorems  \ref{thm:critical} and  \ref{thm:convex} to apply, we must constrain our parameters to 
\begin{equation*}
0 < \gamma < 1 , \quad \gamma-1\leq\beta\leq \gamma.
\end{equation*} }
{We consider ten different setup cases which are provided in Table 1. The first five setups monitor the effect of $\beta$ while the last fiver monitor the parameter $\gamma$. To see the effect of the non-constant step-size $h_n$ we also test a setup where $\beta=0$ resulting in a constant step-size. We set the initial step-size as $h_0=0.5$ and the initial inflation factor as $\alpha_0 = 0.2$. From our analysis, Theorem \ref{thm:convex} suggests that taking a nonzero $\beta$ can result in a faster, and improved performance. 
For the later five setups (Theorem \ref{thm:convex} scenario 2) suggests 
 increasing $\gamma$ with a fixed $\beta$ leads to worse performance. 
As a comparison, we also implement the original vanilla TEKI, which is labeled as VTEKI in the figures and tables. The stepsize is fixed as $h=0.5$ and there is no covariance inflation. We compare our results with the case of no additive covariance inflation, i.e. $\alpha_0 =0$ from the inflation formula \eqref{eq:inflat}, which we label as VTEKI (vanilla Tikhonov EKI).}
{
\begin{table}[h!]
  \begin{center}
    \begin{tabular}{|c|cc|c|cc|c}
       \hline
     \textbf{Setup} & $\beta$ & $\gamma$ &\textbf{Setup}  & $\beta$ & $\gamma$   \\ 
      \hline
 1  & 0 & 0.9 &6  & 0.2 & 0.2\\  
    2  & 0.2 & 0.9 &7 & 0.2 & 0.3\\  
       3  & 0.4 & 0.9 &8  & 0.2 & 0.5\\  
       4  & 0.6 & 0.9 &9  & 0.2 & 0.7\\  
        5  & 0.8 & 0.9 &10  & 0.2 & 0.9\\  
   \hline 
    \end{tabular}
    \bigskip
        \caption{Setups of EKI based on parameter choice for $\beta$ and $\gamma$.}
              \label{table:diff_cases}
  \end{center}
\end{table}
 The effect of modifying $K$ and $J$ has been well documented in \cite{CIRS17,MAI16}. 
 {For the implementation of our method we will use the transform ensemble Kalman filter \cite{BEM01}. } Note that our model also requires modification for additive covariance inflation, such a modification is documented similarly in \cite{HW05,RCB15}. {As mentioned our work is not restricted to a particular ESRF methodology, of which further details  can be found in \cite{CR13}}. We set our initial covariance inflation as $\alpha_0=0.2$. {The iterative procedure is described by Algorithm \ref{alg:1} in the appendix}. For all test experiments we place the regularization parameter as ${\lambda}$=2. For the noise of the inverse problem we set $\Gamma= \iota^2I$ with $\iota=0.01$. We set the iterative model to run for $n=23$ iterations. }

{We pick the algorithm parameters so that the assumptions of the theoretical results are met. 
The first assumption we made is that the ensemble size must be greater than the dimension of $u$, i.e.
$K > d_u.$ In particular, our ensemble size for each test problem is fixed as $K=\{50,50,200\}$, where the dimension for each test problem is $d_u = \{40, 40, 150\}$. We also choose 
\[
0 < \gamma < 1 , \quad \gamma-1\leq\beta\leq \gamma,
\]
for the inflation factor in Theorem \ref{thm:precision} and Theorem \ref{thm:convex}. We will choose different test cases from Table \ref{table:diff_cases} to see the effect of modifying such parameters on the convergence. 
Finally, in Proposition \ref{prop:ESRF} we assume that $\mathcal{G}(u)$ takes value \textbf{0}, either through the mollification \eqref{eq:moll}, or by choosing $\|u\| \geq M +1$ large enough. When it is chosen one or two magnitude larger than the typical solution $\|u\|$, the mollified version is essentially the same as the one without. For our experiments we have tested this for $M=1000$,  and $C=2000$, however we  emphasize that from our observation, it is mainly a technicality used in the theoretical discussion. We finally note that all these conditions can be easily checked through numerical verification.}

\subsection{{Gauss--Newton implementation}}

{
Section \ref{sec:GN} reveals the connection between EKI and the Gauss--Newton (GN) method. We also implement GN on the test problems mentioned above as a comparison with the EKI methods. Directly implementing \eqref{eqn:GNupdate} on $G(v)=\Gamma_+^{-1/2}\calH$, we find that the GN iterates are given by
\[
v_{n+1}=v_n-[J_n^T\Gamma_+^{-1}J_n]^{-1}J_n\Gamma^{-1/2}_+(\calH(v_n)-z).
\]
%
where $J_n:=\nabla \calH(v_n)$ denotes the Jacobian of the forward map $\mathcal{H}$. It is important to note that GN method relies on the computation of the derivatives, while EKI methods are derivative-free.}

{Since the test problems we consider here are relatively simply, their Jacobian matrices have explicit formulations. For the L96 model, note that given any ODE  $v(t)=v(t,u)$ described by a initial value problem
\[
\frac{d}{dt} v(t)=\Phi(v(t)),\quad v(0)=u,
\]
its Jacobian $Q(t)=\nabla_u v(t,u)$ is given by a matrix-valued ODE:
\[
\frac{d}{dt}Q(t)=\nabla \Phi(v(t))Q(t),\quad Q(0)=I_{d_u}. 
\]
The value of $Q(t)$ can be obtained by ODE methods such as fourth order Runge--Kutta.
With the observation map $\mathcal H(u)=[G v(t,u);u]$, the composite Jacobian can be obtained through $\nabla \calH (v)=[GQ(t);I]$. For the Darcy flow equations, the Jacobian can be obtained through a common approach known as the adjoint method or equation. The adjoint equation can be derived through the theory of Lagrangian multipliers \cite{NW06, SAN05}.
The particular formulation of $J_n$ can be found in \cite{ORL08} section 9.2. We avoid giving out the technical details, because they are long and a digression from our main interest in this paper.
}

\subsection{{Layout of numerical results}}

{Our numerical results will consist of a number of measurements. We will first plot the reconstruction of the underlying truth for each experiment. This will be done for the first five test cases of each model. We will then proceed to plot the relative errors of each test case, where we monitor (i) the relative $l_2$ error of the mean $m_N$ with respect to the truth $m^*$, i.e.
\begin{equation*}
\frac{\|m_N - m^{*}\|_{l_2}}{\|m^*\|_{l_2}},
\end{equation*}
and  (ii) the relative error evaluated in the loss function $\ell$, i.e. $\ell(m_N)$. The relative error plots will also have corresponding tables of the relative error at the end of the experiment for each model as well as the computational cost. For each test model we will compare all test cases with the GN method. These will be shown in the reconstruction and relative error plots.}


\subsection{Inversion results}
{Figures \ref{fig:Lorenz_1}, \ref{fig:PDE} and \ref{fig:reconstruction_2D} demonstrate the reconstruction of the unknown for each model problem. To highlight the effect of the non-constant step-size and variance inflation we plotted the first 5 setups and the case of VTEKI. As seen from the Figures \ref{fig:Lorenz_1} and \ref{fig:PDE},  as we increase $\beta$ we tend to get a better reconstruction which is closer to the truth. A more in-depth analysis of the effect of the step-size and covariance inflation is provided in Tables \ref{table:cases1}, \ref{table:cases2} and \ref{table:cases3}, where we consider the different setups. Throughout each model we notice that as we increase $\beta$ and decrease $\gamma$ we see an improvement in learning the unknown. {The tables also demonstrate the computational cost associated with  TEKI.}

This is verified further through Figures \ref{fig:error_L96}, \ref{fig:error_pde} and  \ref{fig:errors_2D} which highlight the relative errors. For the error plots we see that the decay of the error is faster than the rate $N^{-\alpha}$. We choose $\alpha=0.049$, which satisfies $\alpha<\frac12+\frac12\beta-\frac12 \gamma$ for the first 5 setups.  Note that the error eventually plateaus. The reason for this is that we plot the error w.r.t. the truth, while the convergence result in Theorem \ref{thm:convex} considers the error w.r.t. the minimizer. For the latter 5 setups, the error plots are similar, so we do not show them here. {We also remark that in all these plots, there are no overfitting occur, as the relative errors decrease with iterations. This is due to the Tikhonov  regularization incorporated \cite{CST18}.}

{With all of our test models we have also implemented the GN method. Overall we find that GN has similar performance with the modified EKI methods in terms of relative errors, and it performs better asymptotically in terms of the loss function $\ell$. This is mainly because it has access to the exact Jacobian. In comparison, EKI is using an approximated descent. Moreover, recall that EKI is optimizing $\ell_{n+1}$ in \eqref{eqn:elln} instead of the loss function $\ell$ itself, so GN has a more direct optimization effect on $\ell$. Finally, even though we found the explicit formulation of Jacobian in the two simple test examples,  GN takes roughly twice as much wall clock time to finish. Finally we see that the GN is only slightly more expensive, but this increases with each numerical experiment. This is due to the cost of algorithm in computing the derivatives.}

Our experiments use different initial ensembles for the 1D problems hence we see different errors for the first iteration. This is to show that the method is robust under different initializations. We have conducted identical experiments which have the same initialization, and we see a similar performance with respect to each setup.}

\begin{figure}[h!]
\centering
\includegraphics[scale=0.35]{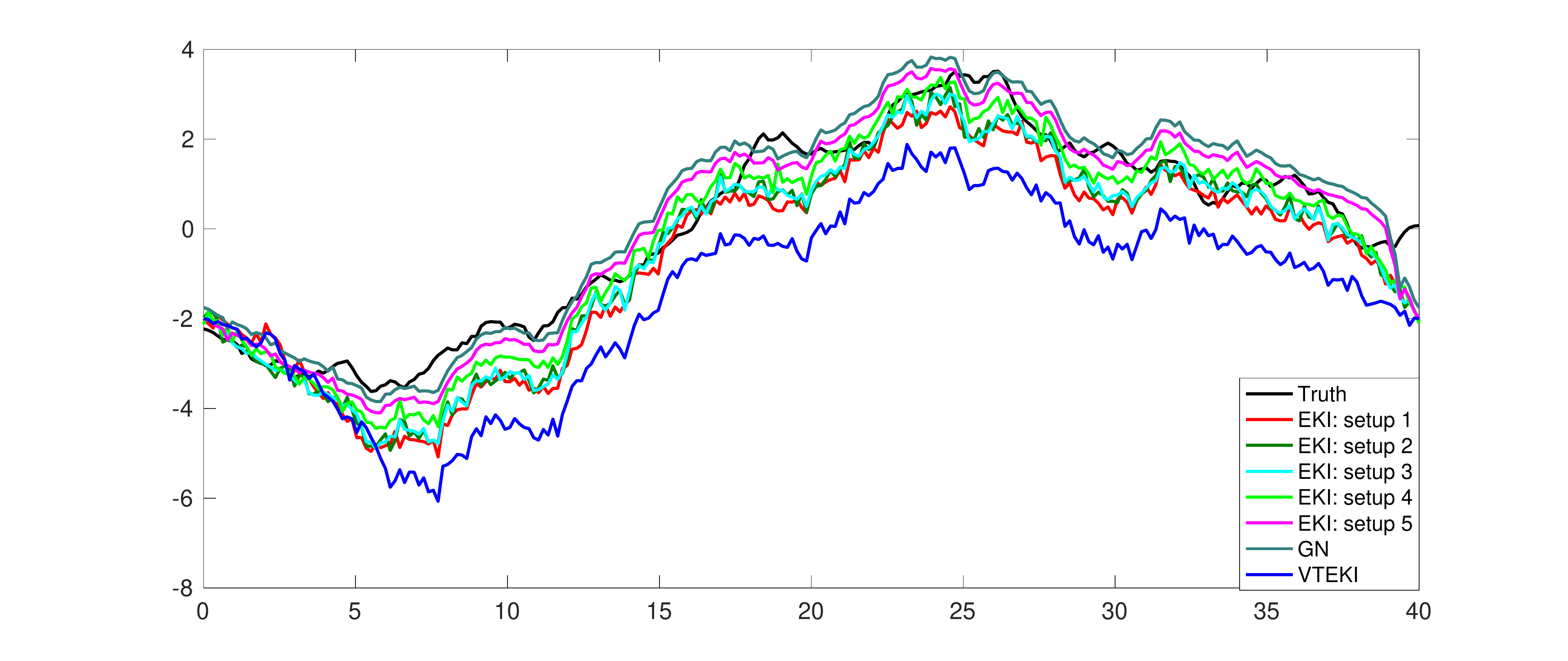}
\caption{Reconstruction of the true initial condition of the L96 model.}
 \label{fig:Lorenz_1}
\end{figure}

\begin{table}[h!]
 \begin{center}
    \begin{tabular}{|c|r|r|r|r}
          \hline
     \textbf{Setup} & Relative error & Loss function &  Time(min)  \\ 
      \hline
      1 &\quad  0.126 &  1.80 &  \textcolor{black}{74}  \\  
    2 & 0.076  & 1.166 &  \textcolor{black}{74} \\  
     3 & 0.073  & 1.61 &  \textcolor{black}{76} \\  
      4 & 0.065  & 1.57 &  \textcolor{black}{77} \\  
      5 & 0.057  & 1.56 &  \textcolor{black}{79} \\  
      6 & 0.065 & 1.59 &  \textcolor{black}{77}  \\  
      7 & 0.068& 1.62 &  \textcolor{black}{78}   \\  
      8 & 0.073 & 1.67 &  \textcolor{black}{76}   \\  
     9 & 0.077 & 1.71 &  \textcolor{black}{74}  \\  
      10 &0.081 &1.73 &  \textcolor{black}{73}  \\  
       VTEKI & 0.183  & 1.84 &  \textcolor{black}{73}  \\  
       GN & 0.042  &  1.48 & \textcolor{black}{80} \\  
               \hline
    \end{tabular}
    \bigskip
        \caption{{Effect of changing $\beta$ and $\gamma$ for the L96 model and the computational cost}.}
              \label{table:cases1}
  \end{center}
\end{table}

\begin{figure}[h!]
\centering
\includegraphics[width=\linewidth]{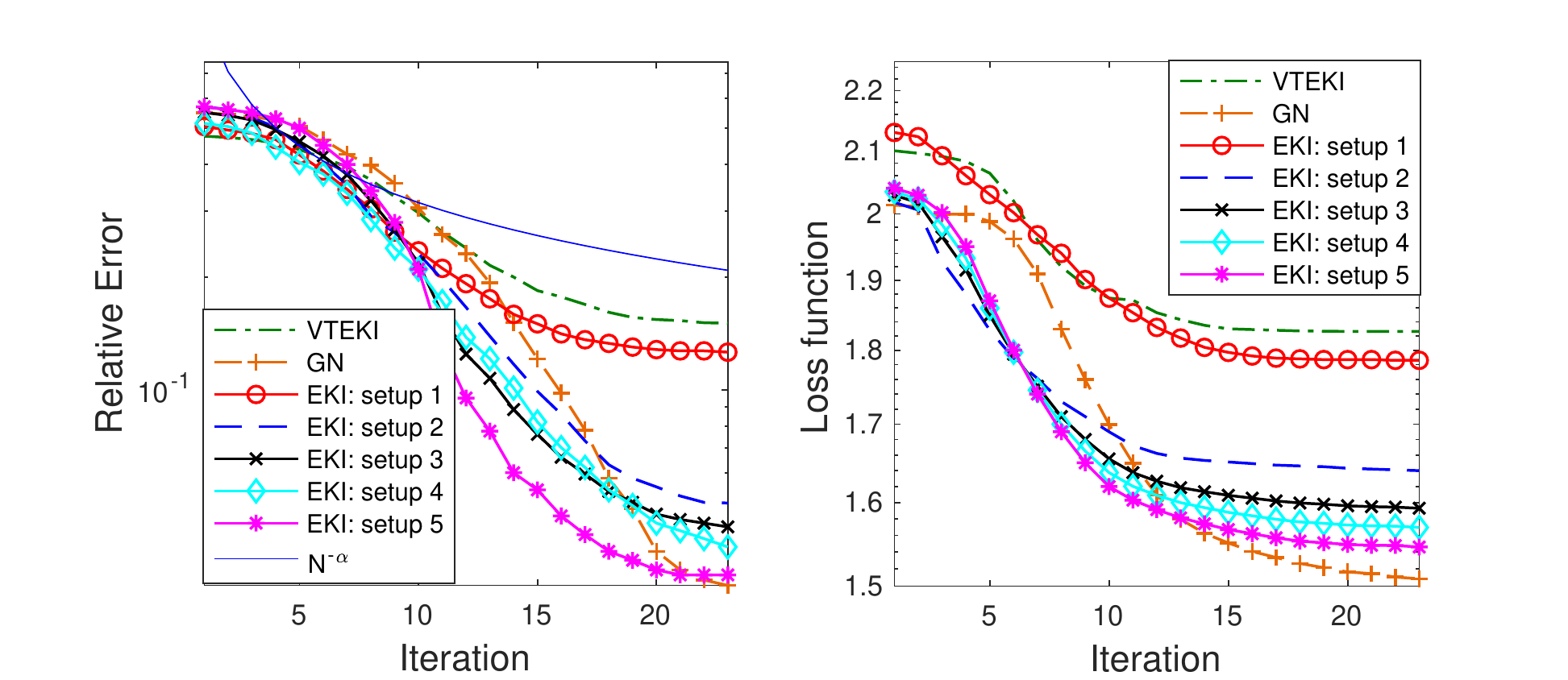}
\caption{Relative error and error in loss function of the L96 model.}
 \label{fig:error_L96}
\end{figure}


\begin{figure}[h!]
\centering
\includegraphics[scale=0.35]{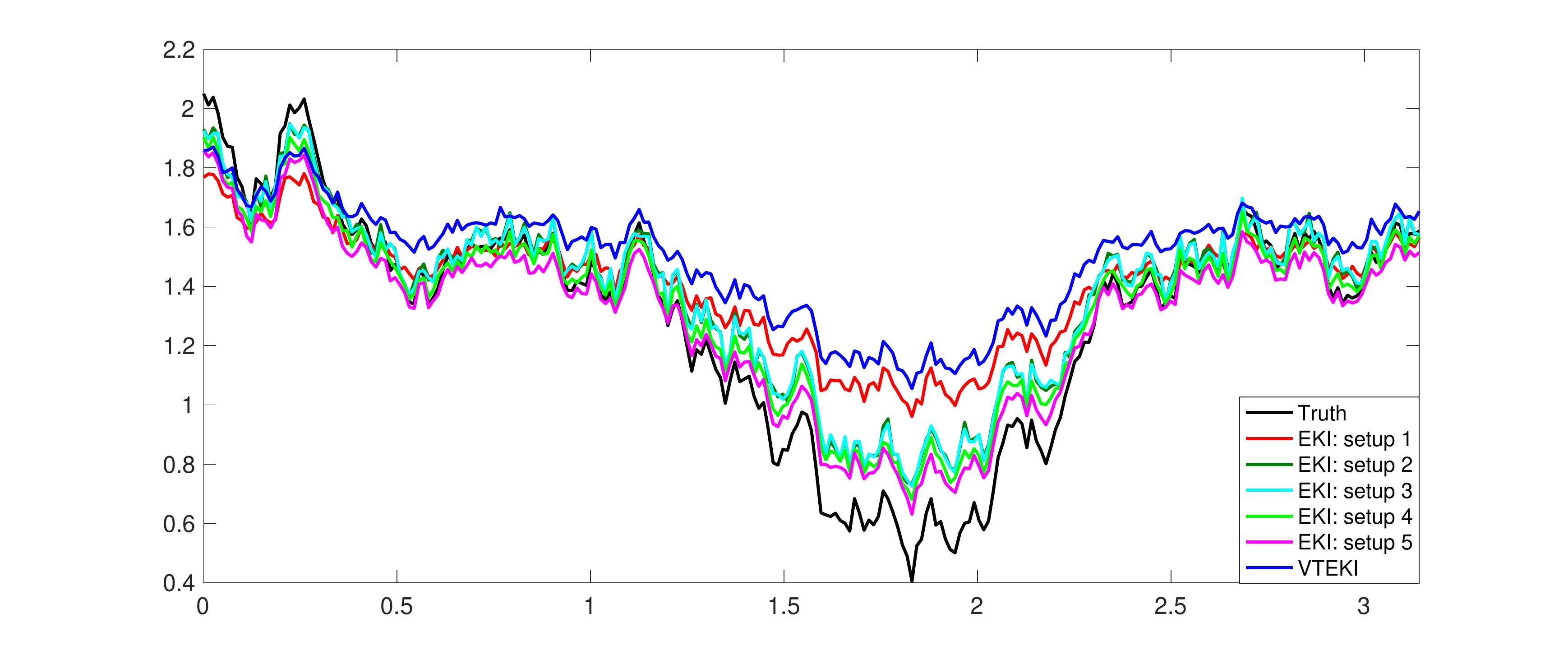}
\caption{Reconstruction of the true initial condition of 1D Darcy flow.}
 \label{fig:PDE}
\end{figure}

\begin{table}[h!]
  \begin{center}
    \begin{tabular}{|c|r|r|r|r}
          \hline
     \textbf{Setup} & Relative error & Loss function &  Time (min)\\ 
     \hline
      1 &\quad  0.112 &  1.63 &  \textcolor{black}{92}\\  
     2 & 0.058  &  1.50 &  \textcolor{black}{94}\\  
       3 & 0.053  & 1.45 &  \textcolor{black}{97} \\  
        4 & 0.047  &  1.42 &  \textcolor{black}{95} \\  
         5 & 0.044  &  1.40 &  \textcolor{black}{98} \\  
        6 & 0.044  &  1.37 &  \textcolor{black}{97} \\ 
        7 & 0.046 & 1.39 &  \textcolor{black}{94} \\
        8 & 0.049  & 1.42 &  \textcolor{black}{96} \\
        9 & 0.057  & 1.48 &  \textcolor{black}{93} \\
        10 & 0.060  &  1.53 & \textcolor{black}{95}  \\
       VTEKI & 0.161  &  1.87 &  \textcolor{black}{91}\\  
       GN & 0.038  & 1.33  & \textcolor{black}{102}\\  
        \hline
  \end{tabular}
    \bigskip
        \caption{{Effect of changing $\beta$ and $\gamma$ for 1D Darcy flow and the computational cost}.}
              \label{table:cases2}
  \end{center}
\end{table}

\begin{figure}[h!]
\centering
\includegraphics[width=\linewidth]{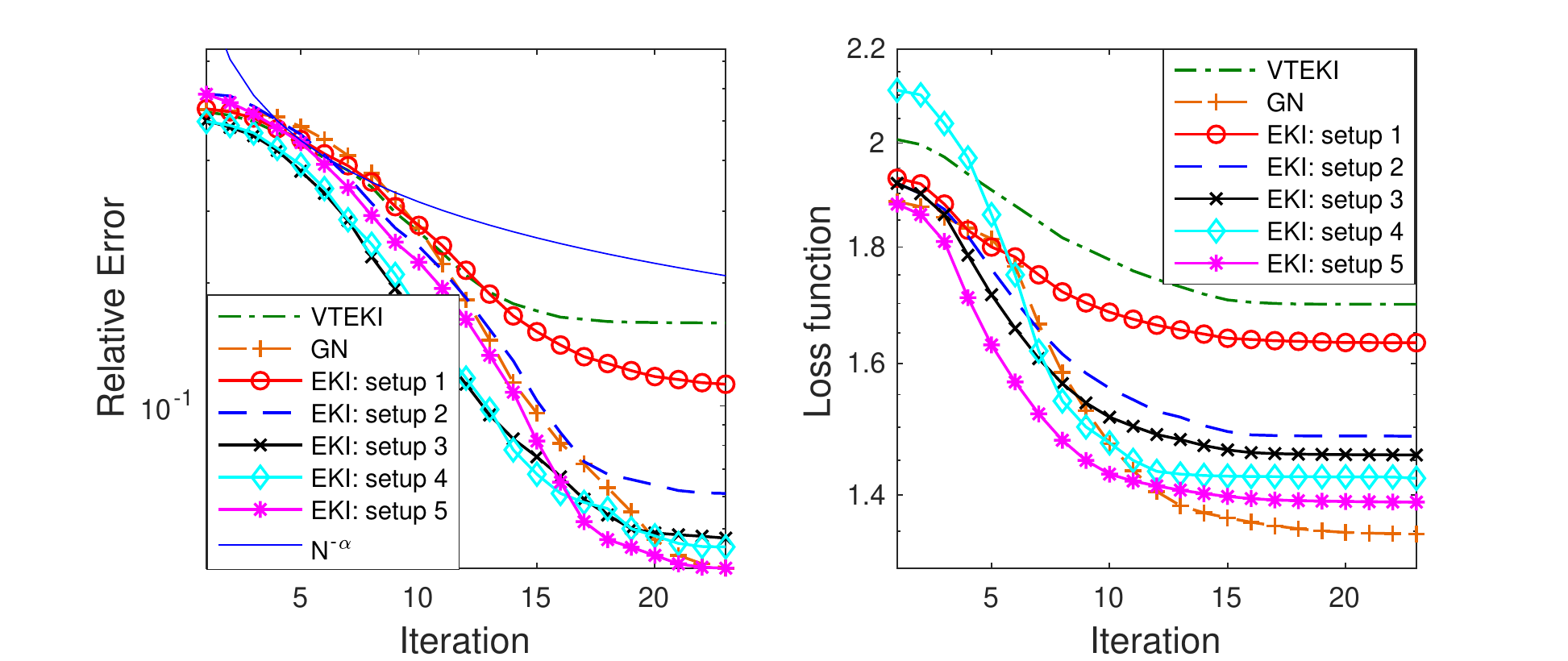}
\caption{Relative error and error in loss function of 1D Darcy flow.}
 \label{fig:error_pde}
\end{figure}

\begin{figure}[h!]
\centering
\includegraphics[scale=0.3]{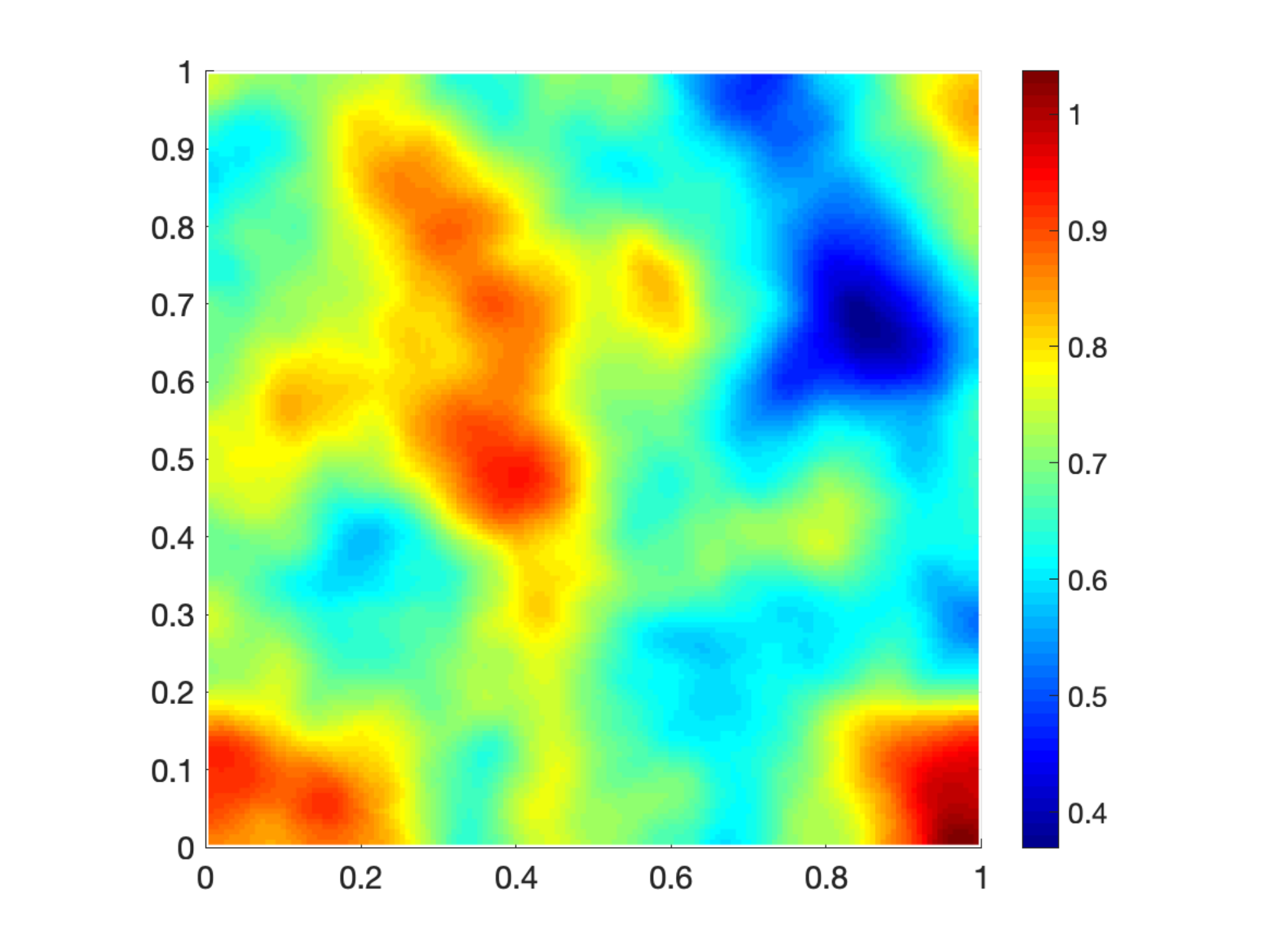}
\caption{Gaussian random field truth.}
 \label{fig:truth2D}
\end{figure}


\begin{figure}[h!]
\centering
\includegraphics[width=\linewidth,trim=1cm 0cm 1cm 0cm,clip]{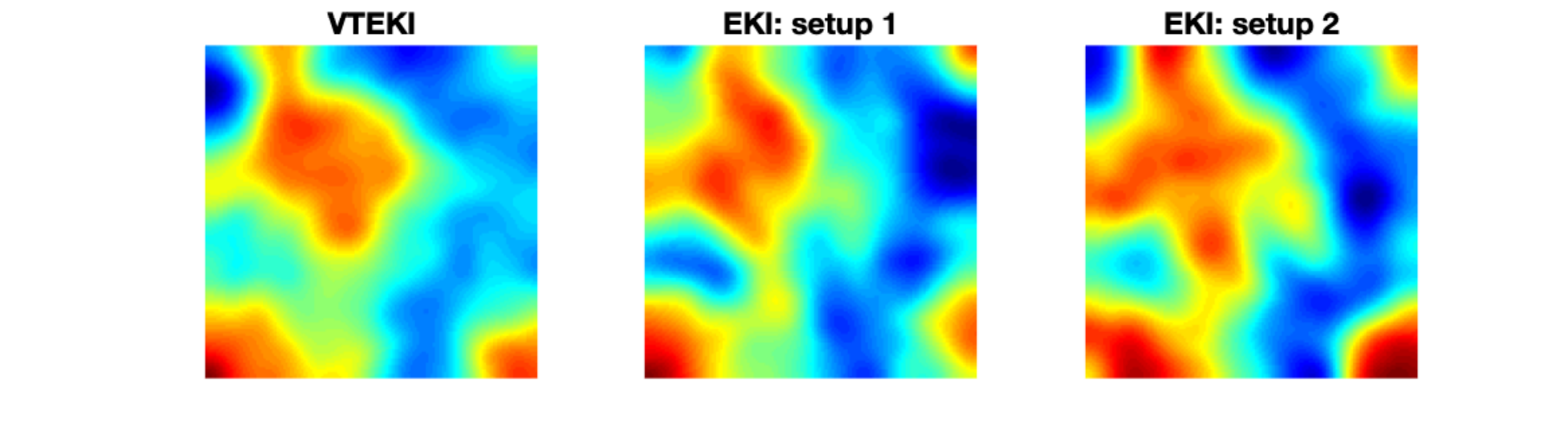}
\includegraphics[width=\linewidth,trim=1cm 0cm 1cm 0cm,clip]{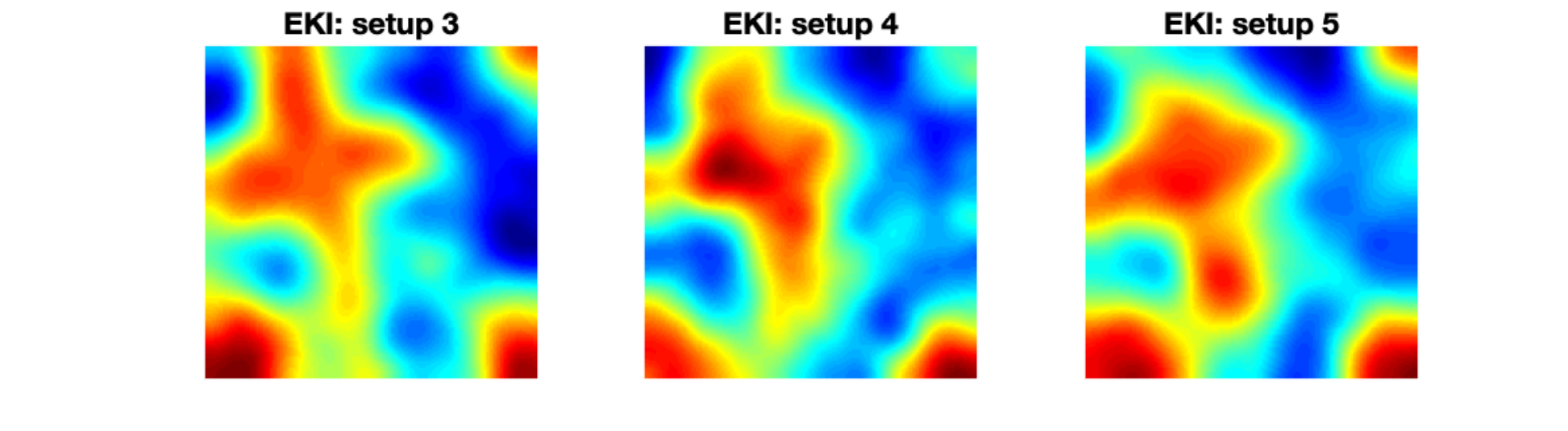}
\includegraphics[width=\linewidth,trim=1cm 0cm 1cm 0cm,clip]{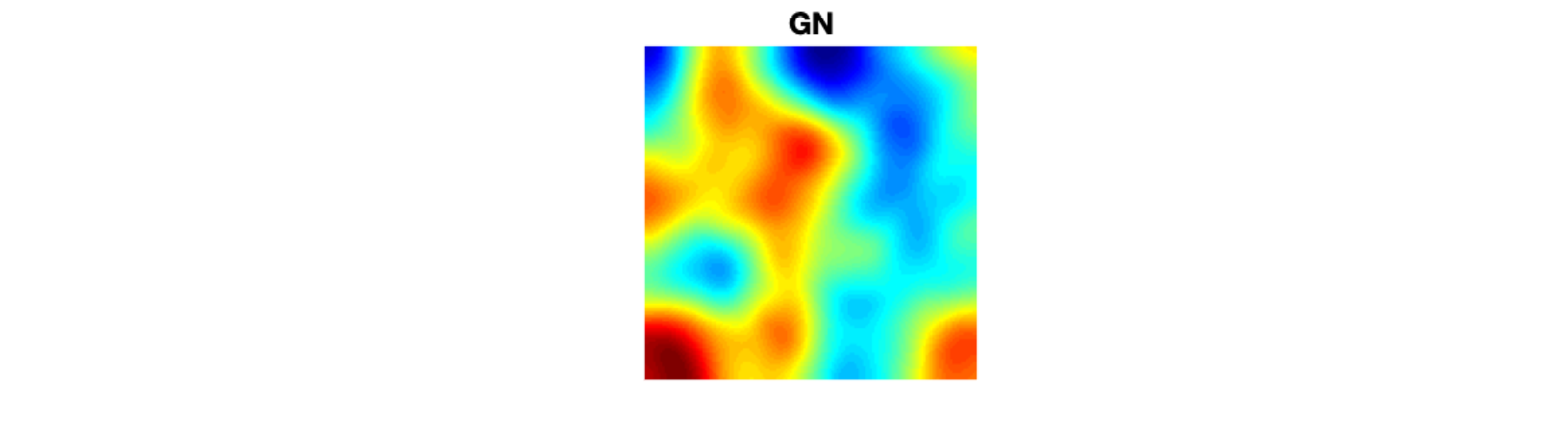}
\caption{Reconstruction of the truth for 2D Darcy flow. Top row. Left: Vanilla TEKI. Centre: EKI setup 1. Right : EKI setup 2. Bottom row. Left:  EKI setup 3. Centre: EKI setup 4. Right: EKI setup 5.}
 \label{fig:reconstruction_2D}
\end{figure}

\begin{table}[h!]
  \begin{center}
    \begin{tabular}{|c|r|r|r|r}
          \hline
     \textbf{Setup} & Relative error & Loss function &  Time (min)\\ 
     \hline
1 &\quad  0.172 &  2.01 & \textcolor{black}{281} \\  
 2 & 0.164  &  1.89  & \textcolor{black}{278} \\  
   3 & 0.158  & 1.84  & \textcolor{black}{274} \\  
 4 & 0.154  &  1.81  & \textcolor{black}{276}\\  
         5 & 0.152  &  1.76  & \textcolor{black}{273} \\  
 6 & 0.154  & 1.64 & \textcolor{black}{275} \\
 7 & 0.156  & 1.67 & \textcolor{black}{278} \\
 8 & 0.161  & 1.73 & \textcolor{black}{275} \\
 9 & 0.164 &  1.79 & \textcolor{black}{279}\\
 10 & 0.167   & 1.89 & \textcolor{black}{282} \\
  VTEKI & 0.208  &  2.06  &\textcolor{black}{279}  \\  
GN & 0.137  & 1.66  & \textcolor{black}{307}\\  
 
                  \hline
    \end{tabular}
    \bigskip
        \caption{{Effect of changing $\beta$ and $\gamma$ for 2D Darcy flow and the computational cost}.}
              \label{table:cases3}
  \end{center}
\end{table}


\begin{figure}[h!]
\centering
\includegraphics[width=\linewidth]{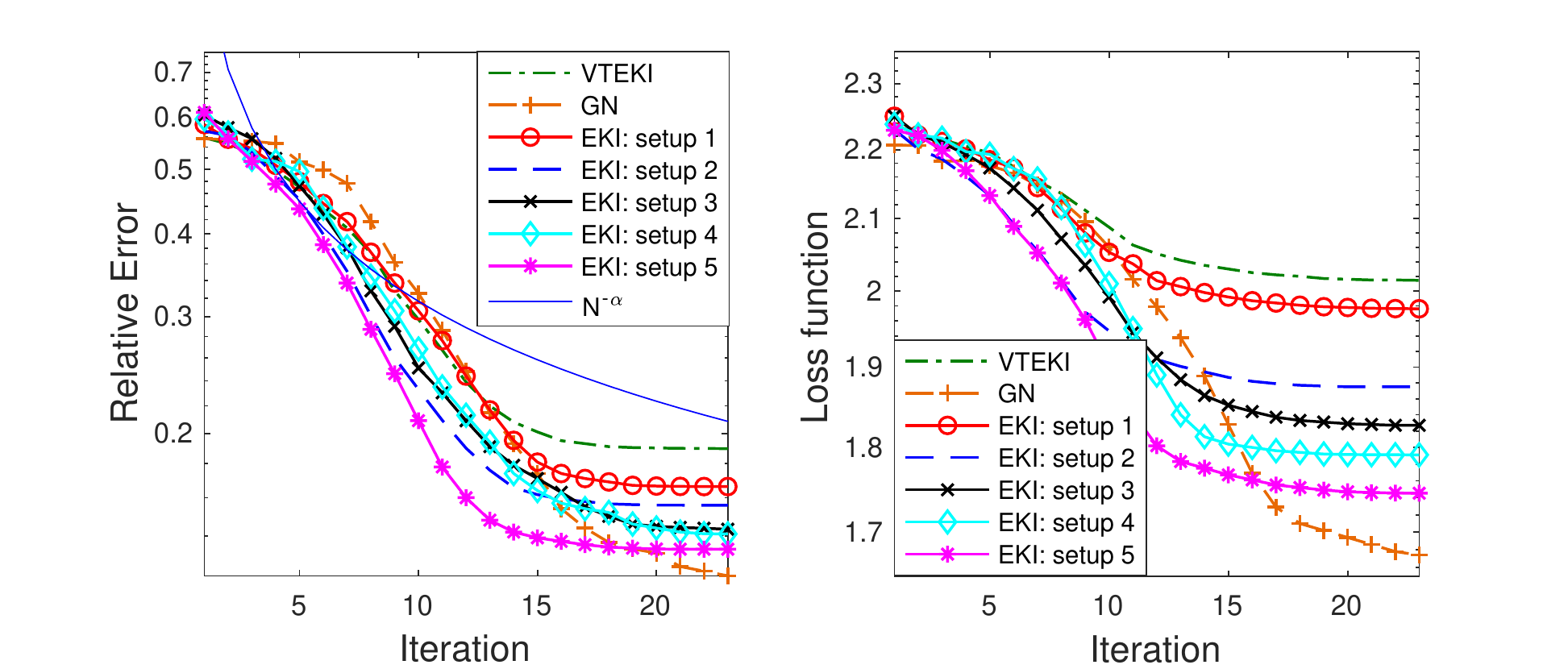}
\caption{Relative error and error in loss function of 2D Darcy flow with EKI.}
 \label{fig:errors_2D}
\end{figure}

\section{Conclusion}
\label{sec:conc}

Ensemble Kaman inversion (EKI) is a derivate-free method used to solve inverse problems. As it formulated in a variational manner, a natural question to ask is how to accelerate its convergence. We aimed to answer this by considering a discrete time formulation of EKI and providing a convergence analysis in  general nonlinear settings. We show that approximate stationary points are attainable in a non-convex setting while global minimizers are attainable in a strongly convex setting. A key insight in our work was the use of filtering techniques such as covariance inflation and adopting an ensemble square-root formulation, as well as ideas of non-constant step-sizes from machine learning. These results were highlighted through different numerical examples which validate the performance improvement of EKI.  By using various inverse problems, we were able to see the consistency among experiments. {Also from the experiments we saw that EKI performed similarly to the GN method, whilst taking less time to run.}

This work promotes various directions to take both in terms of algorithmic efficiency and also further theory for EKI. For the algorithmic efficiency one could consider understanding the relationship between EKI and common optimization techniques which include  Nesterov acceleration \cite{YN04}, momentum and stochastic gradient descent \cite{CLTZ16,VP17}, where initial work has been done in this in the context of machine learning \cite{HLR18,KS18}. {As our results hold for a nonlinear operator in discrete time, a natural direction is to to investigate similar properties in the noisy continuous-time setting, similar to that of \cite{BSWW19}}. {An other direction to consider is the adaptive learning of the regularization parameter ${\lambda}$, which would optimize the effect of the regularization}. {Finally as EKI is viewed as an optimizer but can use techniques from Bayesian methodologies, it would be of interest to compare EKI with commonly used Bayesian methods for optimization. These include simulated annealing as discussed, but also hybrid methods such as parallel tempering and variants of sequential Monte Carlo. These and other areas will considered for future work.}

\section*{Acknowledgments}

NKC acknowledges a Singapore Ministry of Education Academic Research Funds Tier 2 grant [MOE2016-T2-2-135]. The research of XTT is supported by the Singapore Ministry of Education Academic Research Funds Tier 1 grant R-146-000-292-114.

\appendix
\section{Convergence of the ensemble}
\label{sec:proof1}
The proof of Theorem \ref{thm:precision} is decomposed into four lemmas below. Note that we can always normalize the problem by consider the transformation $\tilde{u}=\Sigma^{-1/2}u$,  so the prior covariance of $\tilde{u}$ is $\mathbf{I}_{d_u}$. Likewise,  the observation model \eqref{eqn:IP} can be transformed to 
\[
\tilde{y}=\Gamma^{-1/2}\calG(u)+\Gamma^{-1} \eta=\widetilde{\calG}(\tilde{u})+\tilde{\eta}.
\]
Therefore without loss of generality, we assume $\Sigma=\mathbf{I}_{d_u}, \Gamma=\bfI_{d_y}$ in most of the analysis below for simplicity. 
\begin{lem}
\label{lem:precisionbound}
Let $\omega_n$ be the minimum eigenvalue of $\Sigma^{1/2}(C^{uu}_{n})^{-1}\Sigma^{1/2}$,  then the sequence satisfies: 
\[
\omega_{n+1}\geq \frac{\omega_n+h_n}{1+\alpha_n^2 (\omega_n+h_n)}. 
\]
\end{lem}
\begin{proof}
Recall the sample covariance update rule is given by 
\[
C^{uu}_{n+1}=C^{uu}_n- C^{up}_n (h^{-1}_n \Gamma_++ C^{pp}_n )^{-1}C_n^{pu}+\alpha^2_n \bfI. 
\]
 To continue, we have a closer look at the cross covariance matrices and denote 
\[
A=\frac1K\sum_{i=1}^K \left(\calG(u^{(i)}_n)-\frac1K\sum_{i=1}^K\calG(u^{(i)}_n)\right) \otimes \left(\calG(u^{(i)}_n)-\frac1K\sum_{i=1}^K\calG(u^{(i)}_n)\right), 
\]
\[
B=\frac1K\sum_{i=1}^K (u^{(i)}_n-m_n)\otimes \left(\calG(u^{(i)}_n)-\frac1K\sum_{i=1}^K\calG(u^{(i)}_n)\right), 
\]
\[
C=\frac1K\sum_{i=1}^K (u^{(i)}_n-m_n)\otimes (u^{(i)}_n-m_n)=C^{uu}_n. 
\]
Then 
\[
C^{up}_n=\begin{bmatrix}
B &C
\end{bmatrix},\quad C^{pp}_n=\begin{bmatrix}
A & B^T\\
B &C
\end{bmatrix}.
\]
For a sufficiently large number $M$, the following holds
\[
C^{pp}_n+h_n^{-1}\Gamma_+=\begin{bmatrix}
A+h_n^{-1} \Gamma & B^T\\
B &C+h_n^{-1}\bfI
\end{bmatrix}\preceq 
\begin{bmatrix}
M \bfI  & B^T\\
B &C+h_n^{-1}\bfI
\end{bmatrix}.
\]
As a consequence, we can apply the block matrix inversion formula \cite{DB05}
\begin{align*}
[C^{pp}_n+h_n^{-1}\Gamma_+]^{-1}&\succeq \begin{bmatrix}
M \bfI  & B^T\\
B &C+h_n^{-1}\bfI
\end{bmatrix}^{-1}\\
&=\begin{bmatrix}
\frac1M \bfI +\frac1{M^2} B^T[C+h_n^{-1}\bfI-\frac1MBB^T ]^{-1}B, & -\frac1M B^T[C+h_n^{-1}\bfI-\frac1MBB^T ]^{-1}\\
-\frac1M [C+h_n^{-1}\bfI-\frac1MBB^T ]^{-1} B, &[C+h_n^{-1}\bfI-\frac1MBB^T ]^{-1}
\end{bmatrix}.
\end{align*}
Since this holds for any sufficiently large $M$, we can let $M\to \infty$, and find that
\[
[C^{pp}_n+h_n^{-1}\Gamma_+]^{-1}\succeq \begin{bmatrix}
\mathbf{0}, & \mathbf{0}\\
\mathbf{0}, &[C+h^{-1}_n\bfI]^{-1}
\end{bmatrix}.
\]
Therefore
\[
C^{up}_n [C^{pp}_n +h_n^{-1}\Gamma_+ ]^{-1}C_n^{pu}\succeq C[C+h_n^{-1}\bfI]^{-1} C=C_n^{uu}[C_n^{uu}+h_n^{-1}\bfI]^{-1} C_n^{uu}.
\]
So we have
\[
C^{uu}_{n+1}\preceq C_n^{uu}-C_n^{uu}[C_n^{uu}+h_n^{-1}\bfI]^{-1} C_n^{uu}+\alpha_n^2\bfI. 
\]
Clearly $C_n^{uu}$ shares the same eigenvectors with $C_n^{uu}-C_n^{uu}[C_n^{uu}+h_n^{-1}\bfI]^{-1} C_n^{uu}+\alpha_n^2\bfI$, while an eigenvalue $c$ for the former corresponds to an eigenvalue $c-c(c+h_n^{-1})c+\alpha_n^2$ for the latter. From this, we find that if  
 $\omega_n$ is the minimum eigenvalue of $[C_n^{uu}]^{-1}$, then the minimum eigenvalue of $[C_{n+1}^{uu}]^{-1}$ will be bounded from below by
\[
\omega_{n+1}\geq \frac{\omega_n+h_n }{1+\alpha_n^2 (\omega_n+h_n)}.
\]
\end{proof}
\begin{lem}
\label{lem:conditionbound}
With $\omega_n$ defined in Lemma \ref{lem:precisionbound}, let $c_n$ be the minimal eigenvalue of $\Sigma^{-1/2}C^{uu}_n\Sigma^{-1/2}$, it satisfies the following recursion:
\[
c_{n+1}\geq \frac{c_n}{\tfrac{M_1^2d_u h_n}{\omega_n}+1}+\alpha_n^2. 
\]
\end{lem}
\begin{proof}
Recall the sample covariance matrix is updated by
\[
C^{uu}_{n+1}=C^{uu}_n- C^{up}_n (h^{-1}_n \bfI+ C^{pp}_n )^{-1}C_n^{pu}+\alpha^2_n \bfI. 
\]
Define the ensemble spread matrices:
\[
S^u_n=\frac{1}{\sqrt{K}}\left[u^{(1)}_n-m_n,\cdots, u^{(K)}_n-m_n\right],\quad
S^p_n=\frac{1}{\sqrt{K}}\left[\calH(u^{(1)}_n)-\Hbar_n,\cdots, \calH(u^{(K)}_n)-\Hbar_n\right].
\]
Note that 
\[
C^{uu}_n=S^u_n (S^u_n)^T,\quad C^{up}_n=S^u_n (S^p_n)^T,\quad C^{pp}_n=S^p_n (S^p_n)^T,
\]
moreover,
\[
C^{up}_n (h_n^{-1} \bfI+ C^{pp}_n )^{-1}C_n^{pu}=S^u_nA [S^u_n]^T,\quad A:=[S^p_n]^T (h_n^{-1} \bfI+ C^{pp}_n )^{-1} S^p_n.
\]
Next, note that by Assumption \ref{aspt:obsLip}, 
\begin{align*}
\text{tr}(C_n^{pp}) &=\frac1K\sum_{i=1}^K \|\calH(u^{(i)}_n)-\Hbar_n\|^2\\
&\leq \frac1K\sum_{i=1}^K \|\calH(u^{(i)}_n)-\calH(m_n)\|^2\leq \frac{M_1^2}K\sum_{i=1}^K \|u^{(i)}_n-m_n\|^2=M_1^2\text{tr}(C_n^{uu}).
\end{align*}
Therefore 
\begin{equation}
\label{tmp:Cpp}
\|C^{pp}_n\|\leq \text{tr}(C_n^{pp})\leq M_1^2\text{tr}(C_n^{uu})\leq M_1^2 d_u \|C_n^{uu}\|\leq M_1^2d_u  \omega_n^{-1},
\end{equation}
which leads to the following
\begin{align*}
A^2&= [S^p_n]^T (h_n^{-1} \bfI+ C^{pp}_n )^{-1} S^p_n  [S^p_n]^T (h_n^{-1} \bfI+ C^{pp}_n )^{-1} S^p_n\\
&=[S^p_n]^T (h_n^{-1} \bfI+ C^{pp}_n )^{-1} C^{pp}_n (h_n^{-1} \bfI+ C^{pp}_n )^{-1} S^p_n\\
&\preceq  \|C^{pp}_n \| [S^p_n]^T (h_n^{-1} \bfI+ C^{pp}_n )^{-1}\bfI (h_n^{-1} \bfI+ C^{pp}_n )^{-1} S^p_n\\
&\preceq  \frac{M_1^2 d_u h_n}{ \omega_n} [S^p_n]^T (h^{-1}_n \bfI+ C^{pp}_n )^{-1} h^{-1}_n\bfI (h^{-1}_n \bfI+ C^{pp}_n )^{-1}S^p_n.
\end{align*}
As a consequence,
\begin{align*}
A&=[S^p_n]^T (h_n^{-1} \bfI+ C^{pp}_n )^{-1} (S^p_n  [S^p_n]^T+h^{-1}_n \bfI )(h_n^{-1} \bfI+ C^{pp}_n )^{-1} S^p_n\\
&=A^2+[S^p_n]^T (h^{-1}_n \bfI+ C^{pp}_n )^{-1} h^{-1}_n\bfI (h^{-1}_n \bfI+ C^{pp}_n )^{-1}S^p_n\succeq \left(1+\frac{ \omega_n}{M_1^2d_u h_n}\right) A^2. 
\end{align*}
This leads us to $A\preceq \frac{\bfI}{1+\omega_n/(h_nM_1^2d_u)}.$
Therefore, by $C^{uu}_n\succeq c_n \bfI$,  
\begin{align*}
C^{uu}_{n+1}&=C^{uu}_n- C^{up}_n (h^{-1}_n \bfI+ C^{pp}_n )^{-1}C_n^{pu}+\alpha^2_n \bfI\\
&=C^{uu}_n- S^u_n A (S^u_n)^T+\alpha^2_n \bfI\\
&\succeq \frac{\omega_n}{M_1^2d_u h_n+\omega_n}C^{uu}_n+\alpha_n^2\bfI\succeq \left(\frac{c_n}{\tfrac{M_1^2d_uh_n}{\omega_n}+1}+\alpha_n^2\right)\bfI. 
\end{align*}
\end{proof}

\begin{lem}
\label{lem:precisionsequence}
With $\omega_n$ defined as in Lemma \ref{lem:precisionbound},  if we let $h_n=h_0n^{\beta}$  and $\alpha_n^2= \alpha_0^2h_0^{-1} n^{2\gamma-\beta-2}$, where $\beta\leq \gamma\leq 1+\beta $,  then there is a $\kappa>0$, such that 
\[
\omega_n\geq  h_0 \kappa n^{1+\beta-\gamma},\quad \forall n\geq 1.
\]
\end{lem}
\begin{proof}
We will prove this lemma with mathematical induction. Since 
\[
\omega_{1}\geq \frac{\omega_0+h_0}{1+\alpha_0^2 (\omega_0+h_0)}\geq \frac{h_0}{1+\alpha_0^2 h_0}>0,
\]
we can find a $\kappa$ so that $w_1\geq h_1\kappa$. Suppose $\omega_n\geq  h_0\kappa n^{1+\beta-\gamma} $, we want to show the inequality holds for $\omega_{n+1}$ as well. By Lemma \ref{lem:precisionbound}, it suffices to show 
\begin{equation}
\label{tmp:prebound}
\kappa h_0(n+1)^{1+\beta-\gamma}\leq \frac{h_0(\kappa +n^{\gamma-1})n^{1+\beta-\gamma}}{1+\alpha_n^2h_0(\kappa n^{1-\gamma}+1)n^{\beta}}.
\end{equation}
Because $1+\beta-\gamma\geq 0$, by Taylor expansion, there is a constant $c>0$ so that 
\[
(n+1)^{1+\beta-\gamma}\leq n^{1+\beta-\gamma}+(1+\beta-\gamma)(n+1)^{\beta-\gamma}\leq
n^{1+\beta-\gamma}+cn^{\beta-\gamma}.
\]
The constant $c$ exists, as $(1+\frac1n)^{\beta-\gamma}$ is bounded for all $n$.  Replace $(n+1)^{1+\beta-\gamma}$ in \eqref{tmp:prebound} with this upper bound,  we need to show
\begin{equation}
\label{tmp:pre2}
\kappa (1+cn^{-1})\leq \frac{\kappa +n^{\gamma-1}}{1+\alpha_n^2h_0(\kappa n^{1-\gamma}+1)n^{\beta}}.
\end{equation}
Plug in our formulation of $\alpha_n$, \eqref{tmp:pre2} is equivalent to 
\[
\kappa (1+cn^{-1})+ n^{2\gamma-2} \alpha_0^2(\kappa n^{1-\gamma}+1) \kappa (1+cn^{-1})\leq \kappa +n^{\gamma-1}. 
\]
Or equivalently, 
\[
\alpha_0^2 (\kappa +n^{\gamma-1}) \kappa (1+cn^{-1})+\kappa cn^{1-\gamma} +\kappa cn^{-\gamma} \leq  1. 
\]
When $1\geq \gamma\geq0$, the left hand decreases with $n$,  so we just need to check $n=1$, which is 
\[
\alpha_0^2(\kappa+1) \kappa (1+c)\leq 1-2\kappa c. 
\]
Clearly we can find a small $\kappa>0$, such that it holds. This completes our proof. 
\end{proof}

\begin{lem}
\label{lem:covariancesequence}
With the choice of $h_n$ and $\alpha_n^2$ from Lemma \ref{lem:precisionsequence}, and the definition of $c_n$ in Lemma \ref{lem:conditionbound}, there is an $\eta> 0$ such that 
\[
c_n\geq \eta n^{-1-\beta+\gamma},\quad \forall n\geq 1. 
\]
\end{lem}

\begin{proof}
We will again show this claim by induction. By Lemma \ref{lem:conditionbound}, $c_1\geq \alpha_0^2$ so we can find an $\eta>0$ such that 
$c_1\geq \eta$.  Suppose $c_n\geq \eta n^{-1-\beta+\gamma}$, by Lemma \ref{lem:conditionbound} and Lemma \ref{lem:precisionsequence}, it suffices to show 
\[
\frac{\eta n^{-1-\beta+\gamma}}{1+\tfrac{M_1^2 d_u}{\kappa  n^{1-\gamma}}}+\alpha_0^2 h_0^{-1} n^{2\gamma-2-\beta}\geq \eta (n+1)^{-1-\beta+\gamma}. 
\]
Note that $\eta (n+1)^{-1-\beta+\gamma}\leq \eta n^{-1-\beta+\gamma},$
 it suffices to show that 
\[
\frac{\eta }{1+\tfrac{M_1^2 d_u}{\kappa  n^{1-\gamma}}}+\alpha_0^2h_0^{-1} n^{\gamma-1}\geq \eta. 
\]
Note the inequality $\frac{1}{1+a}\geq 1-a$ holds for all $a>0$, so it suffices for us to show
\[
\eta-\tfrac{M_1^2 d_u}{\kappa }n^{-1+\gamma}\eta+\alpha_0^2h_0^{-1} n^{\gamma-1}\geq \eta, 
\]
or equivalently
\[
\alpha_0^2h_0^{-1}\geq \eta \tfrac{M_1^2 d_u}{\kappa  }. 
\]
This can be done by  choosing a small $\eta$. 
\end{proof}

\section{Connection with Gauss--Newton}
\label{sec:proof2}
\begin{proof}[Proof of Proposition \ref{prop:GN}]
We introduce the following notations 
\[
\Delta u_n^{(i)}=u_n^{(i)}-m_n,\quad J_n:=\nabla \calH(m_n),\quad r_{n}^{(i)}:=\calH(u^{(i)}_n)-\calH(m_n)-J_n\Delta u^{(i)}_n,
\] 
\[
\Delta \calH(u^{(i)}_n):=\calH(u^{(i)}_n)-\Hbar_n=J_n\Delta u^{(i)}_n+\Delta r_{n}^{(i)},\quad \Delta r_{n}^{(i)}=r_n^{(i)}-\frac1K\sum_{j=1}^K r_n^{(j)}.
\]
Then the ensemble covariance matrices can be written as
\begin{align*}
C^{pp}_n&=\frac1K \sum_{i=1}^K\Delta \calH(u^{(i)}_n)\otimes\Delta \calH(u^{(i)}_n)\\
&=J_n C^{uu}_n J_n^T+J_n C^{ur}_n+C^{ru}_nJ_n^T +C^{rr}_n, 
\end{align*}
and
\begin{align*}
C^{up}_n=\frac1K \sum_{i=1}^K\Delta u^{(i)}_n\otimes\Delta \calH(u^{(i)}_n)=C^{uu}_nJ_n^T+C^{ur}_n,
\end{align*}
where
\[
C^{ur}_n=\frac1K\sum_{i=1}^K\Delta u^{(i)}_n\otimes \Delta r^{(i)}_n, \quad
C^{rr}_n=\frac1K\sum_{i=1}^K\Delta r^{(i)}_n\otimes \Delta r^{(i)}_n. 
\]
Recall the movements from EKI and Gauss--Newton are 
\[
\Delta_n:=C^{up}_n( C^{pp}_n+h^{-1}_n\bfI)^{-1}(z-\calH(m_n)),\quad G_n=C^{uu}_nJ_n^T ( J_nC^{uu}_nJ_n^T+ h_n^{-1}\bfI)^{-1} (z-\calH(m_n)).
\]
Given these formulations,  we will achieve our claim by showing
\begin{equation}
\label{tmp:Cndiff}
\|C_n^{up}-C_n^{uu} J_n^T\|=\|C_n^{ur}\|\leq d_uM_2\sqrt{K  }\|C^{uu}_n\|^\frac{3}{2},
\end{equation}
and
\begin{align}
\notag
\|( C^{pp}_n+h^{-1}_n\bfI)^{-1}-&( J_nC^{uu}_nJ_n^T+ h_n^{-1}\bfI)^{-1}\|\\
\label{tmp:Cinvdiff}
&\leq 4h_n^2(M_2^2 Kd_u^2 \|C_n^{uu}\|^2+2M_1 M_2 d_u\sqrt{K} \|C_n^{uu}\|^{\frac{3}{2}}). 
\end{align}
Once we have \eqref{tmp:Cndiff} and \eqref{tmp:Cinvdiff},  Proposition \ref{prop:GN} can be  proved by
\begin{align*}
\|G_n-\Delta_n\|\leq  &\|C_n^{up}-C_n^{uu} J_n^T\|\| (C_n^{pp}+h_n^{-1}\bfI)^{-1}\| \|z-\calH(m_n)\|\\
&+\|C_n^{uu}\| \|J_n^T\|\|( C^{pp}_n+h^{-1}_n\bfI)^{-1}-( J_nC^{uu}_nJ_n^T+ h_n^{-1}\bfI)^{-1}\|\|z-\calH(m_n)\|,
\end{align*}
along with $\| (C_n^{pp}+h_n^{-1}\bfI)^{-1}\|\leq h_n$. Note that  the high order terms of $\|C_n^{uu}\|$, which tend to zero by Theorem \ref{thm:precision}, are suppressed in Proposition  \ref{prop:GN}. 

For \eqref{tmp:Cndiff}, we note that given two uninorm vectors $w$ and $v$
\begin{align*}
w^TC_n^{ur} v&=\frac1K\sum_{i=1}^K\langle \Delta u^{(i)}_n, w\rangle\langle v, \Delta r^{(i)}_n\rangle\\
&\leq \frac1K\sqrt{\sum_{i=1}^K\langle \Delta u^{(i)}_n, w\rangle^2 \sum_{i=1}^K\langle v, \Delta r^{(i)}_n\rangle^2}=\sqrt{w^TC^{uu}_n w v^TC^{rr}_n v}.
\end{align*}
In other words we have $\|C^{ur}_n\|\leq \sqrt{\|C^{uu}_n\|\|C^{rr}_n\|}$. Next, notice by Taylor expansion and Assumption \ref{aspt:obsLip},  $\|r_n^{(i)}\|\leq M_2 \|\Delta u_n^{(i)}\|^2$. Therefore
\begin{align*}
\text{tr}(C^{rr}_n)=\frac1K \sum_{i=1}^K \|\Delta r_n^{(i)}\|^2\leq\frac1K \sum_{i=1}^K \| r_n^{(i)}\|^2\leq \frac{M^2_2}{K} \sum_{i=1}^K\|\Delta u_n^{(i)}\|^4
\leq \frac{M^2_2}{K} \left(\sum_{i=1}^K\|\Delta u_n^{(i)}\|^2\right)^2.
\end{align*}
Note that $\sum_{i=1}^K\|\Delta u_n^{(i)}\|^2=K\text{tr}(C^{uu}_n),$ we have
\[
\|C^{rr}_n\|\leq \text{tr}(C^{rr}_n)\leq M^2_2K \text{tr}(C^{uu}_n)^2\leq M^2_2Kd_u^2 \|C^{uu}_n\|^2. 
\]
Combining all the estimates we have \eqref{tmp:Cndiff}.

Next we turn to \eqref{tmp:Cinvdiff}, where by Lemma A.1 in \cite{MTM19} we have
\[
\|[ J_nC^{uu}_nJ_n^T+ h_n^{-1}\bfI]^{-1}-[C^{pp}_n+h_n^{-1}\bfI]^{-1} \|\leq \frac{ \|J_nC^{uu}_nJ_n^T-C^{pp}_n \|\|[ J_nC^{uu}_nJ_n^T+ h_n^{-1}\bfI]^{-1}\|^2}{1-\|J_nC^{uu}_nJ_n^T-C^{pp}_n \|\|[ J_nC^{uu}_nJ_n^T+ h_n^{-1}\bfI]^{-1}\|}.
\]
Since $\|[ J_nC^{uu}_nJ_n^T+ h_n^{-1}\bfI]^{-1}\|\leq h_n$, it suffices to show 
\begin{align*}
\|J_nC^{uu}_nJ_n^T-C^{pp}_n \|&=\|J_n C^{ur}_n+C^{ru}_nJ_n^T +C^{rr}_n\|\\
&\leq 2M_1 \|C^{ur}_n\| +\|C^{rr}_n\|\leq M_2^2 Kd_u^2 \|C_n^{uu}\|^2+2M_1 M_2 d_u\sqrt{K} \|C_n^{uu}\|^{\frac{3}{2}}. 
\end{align*}
This concludes our proof. 
\end{proof}
\section{Iterative descent for EKI}
\begin{proof}[Proof of Proposition \ref{prop:perturb}]
Recall that the mean movement suggested by EKI and Gauss--Newton are 
\[
\Delta_n=C_n^{up}(C_n^{pp}+h_n^{-1}\bfI)^{-1} (z-\calH(m_n)), \quad G_n=C^{uu}_nJ_n^T ( J_nC^{uu}_nJ_n^T+ h_n^{-1}\bfI)^{-1} (z-\calH(m_n)). 
\]
The mean update is $m_{n+1}=m_n+\Delta_n$. By applying the Taylor expansion, we have
\[
\calH(m_{n+1})-z=\calH(m_n)-z+J_n \Delta_n+R_{2,n},\quad R_{2,n}:=\frac12 \Delta_n^T\nabla^2 \calH(m_n+s_n \Delta_n)\Delta_n.
\]
Here $s_n$ is a certain number in the interval $[0,1]$.  The square $l_2$ norm of above is given by 
\begin{align*}
\ell(m_{n+1})=&\|\calH(m_n)-z\|^2- 2\langle \calH(m_n)-z, J_n \Delta_n\rangle
+\|J_n \Delta_n\|^2\\
&+\|R_{2,n}\|^2+ 2\langle \calH(m_n)-z+J_n \Delta_n, R_{2,n}\rangle\\
=&\ell(m_n)- 2\langle \calH(m_n)-z, J_n G_n\rangle\\
&+ 2\langle \calH(m_n)-z, J_n (G_n-\Delta_n)\rangle+\|J_n \Delta_n\|^2+\|R_{2,n}\|^2 \\ &+ 2\langle \calH(m_n)-z+J_n \Delta_n, R_{2,n}\rangle.
\end{align*}
Recall that before Proposition \ref{prop:perturb}, we showed that
\[
\langle \calH(m_n)-z, J_n G_n\rangle=\|J_n^T (\calH(m_n)-z)\|^2_{(h_nC^{uu}_n)^{-1}+J_n^T\Gamma^{-1}_+ J_n}.
\]
So it remains to bound the residual
\begin{equation}
\label{tmp:Rn}
R_n=\|J_n \Delta_n\|^2+\|R_{2,n}\|^2+ 2\langle \calH(m_n)-z+J_n \Delta_n, R_{2,n}\rangle+2\langle \calH(m_n)-z, J_n (G_n-\Delta_n)\rangle.  
\end{equation}
By Assumption \ref{aspt:obsLip}, 
\[
\|J_n \Delta_n\|^2\leq M_1^2 \|\Delta_n\|^2,
\]
and recall that
\[
\|R_{2,n}\|\leq \|\frac12 \nabla^2 \calH(m_n+s_n \Delta_n)\|\|\Delta_n\|^2\leq M_2 \|\Delta_n\|^2. 
\]
By \eqref{tmp:Cndiff}, $\|C_{n}^{up}\|\leq \sqrt{\|C^{pp}_n\|\|C^{uu}_n\|}\leq M_1\sqrt{d_u}\|C^{uu}_n\|$, so
\[
\|\Delta_n\|\leq h_n\|C_{n}^{up}\|\|z-\calH(m_n)\|\leq M_1 h_n \sqrt{d_u} \|C^{uu}_n\| \|z-\calH(m_n)\|.
\]
Then
\begin{align*}
2\langle \calH(m_n)-z+J_n \Delta_n, R_{2,n}\rangle&\leq 2\|\calH(m_n)-z\|\| R_{2,n}\|+2\|J_n \Delta_n\|\|R_{2,n}\|\\
&\leq 2\|\calH(m_n)-z\|\| R_{2,n}\|+\|J_n \Delta_n\|^2+\|R_{2,n}\|^2. 
\end{align*}
Lastly, we have the following by Proposition \ref{prop:GN},
\begin{align*}
2\langle \calH(m_n)-z, J_n (G_n-\Delta_n)\rangle&\leq 2M_1M_3 h_n \|\calH(m_n)-z\| \|C_n^{uu}\|^{\frac{3}{2}} .
\end{align*}
Replacing each term of \eqref{tmp:Rn} with an upper bound developed above, we find
\[
R_n\leq 2M_1M_3h_n \|C_n^{uu}\|^{\frac{3}{2}}  \|z-\calH(m_n)\|^2+2M_1^2 \|\Delta_n\|^2+2M_2^2 \|\Delta_n\|^4+ 2\|\calH(m_n)-z\| \|\Delta_n\|^2,
\]
so there is a constant $M_4$ such that
\[
R_n\leq M_4 h_n \|C_n^{uu}\|^{\frac{3}{2}}\max\{\|z-\calH(m_n)\|^4,1\}. 
\]
\end{proof}

\section{Boundedness of the iterates}

\begin{proof}[Proof of Proposition \ref{prop:ESRF}]
We pick an $n_0$, so that by Theorem \ref{thm:precision}, when $n\geq n_0$
\[
h_n \|C_n^{uu}\| M^2_1\sqrt{d_u}<1, \quad Kd_u\|C_n^{uu}\|\leq 1. 
\]
Note that $\|C^{pu}_n\|\leq\sqrt{\|C^{pp}_n\|\|C^{uu}_n\| }$, and by Assumption \ref{aspt:obsLip}, 
\begin{align*}
\|C^{pp}_n\|\leq \text{tr}(C^{pp}_n)=\frac1K \sum_{i=1}^K \|\Delta \calH(u_n^{(i)})\|^2\leq\frac {M_1^2}K \sum_{i=1}^K \|\Delta u_n^{(i)}\|^2= M_1^2 \tr(C^{uu}_n)\leq M_1^2 d_u\|C^{uu}_n\|. 
\end{align*}
If $\|m_n\|\leq M+1$, then
\begin{align*}
\|m_{n+1}\|&\leq \|m_n\|+\|C^{up}_n( C^{pp}_n+h^{-1}_n\bfI)^{-1}\|(z-\calH(m_n))\|\\
&\leq \|m_n\|+ M_1 \sqrt{d_u} h_n\|C^{uu}_n\|(\|z\|+M_1 \|m_n\|)\\
&\leq (1+h_n \|C^{uu}_n\| M^2_1 \sqrt{d_u})\|m_n\|+ \sqrt{d_u} h_n\|C^{uu}_n\|\|z\|\leq 2M+\|z\|+2. 
\end{align*}
If $m_n\geq M+1$, then due to 
\[
\sum_{i=1}^K\|\Delta u_n^{(i)}\|^2=K\tr(C^{uu}_n)\leq  Kd_u\|C_n^{uu}\|\leq 1,
\]
we have $\|u_n^{(i)}\|\geq M$, by our assumption $\calG(u_n^{(i)})=\mathbf{0}$. This leads to 
\begin{align*}
m_{n+1}&=m_n-C^{up}_n( C^{pp}_n+h^{-1}_n\bfI)^{-1}(z-\calH(m_n))\\
&=m_n- \begin{bmatrix} \mathbf{0}, C^{uu}_n\end{bmatrix} \begin{bmatrix} h_n\bfI , & \mathbf{0}\\ \mathbf{0}, &[C^{uu}_n+h_n\bfI]^{-1}\end{bmatrix}\begin{bmatrix} \mathbf{0}\\ m_n \end{bmatrix}\\
&=m_n-C^{uu}_n[C^{uu}_n+h_n\bfI]^{-1}m_n \\
&=[h_n^{-1}C_n^{uu}+\bfI]^{-1}m_n,
\end{align*}
using that $[h_n^{-1}C_n^{uu}+\bfI]^{-1}\preceq \bfI$,  $\|m_{n+1}\|\leq \|m_n\|$.  As a combination of the two cases, we see that the EKI mean sequence will be decreasing if $m_n\geq M+1$ and $n\geq n_0$. 
\end{proof}

\section{Convergence to minimums}

\begin{proof}[Proof of Theorem \ref{thm:critical}]
Denote the minimum eigenvalue of $\Gamma_+$ as $\gamma_m$ and the maximum eigenvalue of $\Sigma$ as $\sigma_m$.
Then by  Theorem \ref{thm:precision}, we have
\[
(h_n C_n^{uu})^{-1}+J_n^T\Gamma^{-1}_+ J_n\preceq h^{-1}_n \kappa^{-1}_m n^{1+\beta-\gamma} \Sigma^{-1}+\gamma^{-1}_m J_n^T J_n\preceq
h_0^{-1}\kappa^{-1}_m n^{1-\gamma}\sigma_m^{-1}+M_1^2 \gamma_m^{-1}. 
\]
Inserting this in Proposition \ref{prop:perturb}, results in 
\begin{align}
\notag
\ell(m_{n+1})\leq &\ell(m_n)-2(h_0^{-1}\kappa^{-1}_m n^{1-\gamma}\sigma_m^{-1}+M_1^2 \gamma_m^{-1})^{-1}\|\nabla \ell (m_n)\|^2\\
\label{tmp:Ldiff1}
&+M_4 h_n \|C_n^{uu}\|^{\frac{3}{2}}\max\{\|z-\calH(m_n)\|^4,1\}.
\end{align}
Since $m_n$ is bounded, by Theorem \ref{thm:precision} we can let  $c=h_0 \kappa_m \sigma_m$  such that when $n$ is larger than a threshold $n_0$,
\[
2(h_0^{-1}\kappa^{-1}_m n^{1-\gamma}\sigma_m^{-1}+M_1^2 \gamma_m^{-1})^{-1}\geq c n^{\gamma-1},\quad \forall n\geq n_0.
\]
Furthermore there is a constant $D_1$ such that 
\[ 
M_4 h_n K \|C_n^{uu}\|^{\frac{3}{2}}\max\{\|z-\calH(m_n)\|^4,1\}\leq D_1 n^{\frac32\gamma-\frac32-\frac12\beta},\quad n\geq n_0.
\]
Therefore \eqref{tmp:Ldiff1} can be simplified as
\begin{equation}
\label{tmp:Ldiff}
\ell(m_{n+1})\leq\ell(m_n)-c n^{\gamma-1}\|\nabla \ell (m_n)\|^2+D_1 n^{\frac32\gamma-\frac32-\frac12\beta}\quad n\geq n_0.
\end{equation}
So summing \eqref{tmp:Ldiff} over  all $n$ between $n_0$ and $N$,
\begin{equation}
\label{tmp:Ldiff2}
\ell(m_{N})\leq \ell(m_{n_0})-c\sum_{n=n_0}^{N-1}n^{\gamma-1}  \|\nabla \ell (m_n)\|^2+D_1+D_1\sum_{n=n_0}^{N-1} n^{\frac32\gamma-\frac32-\frac12\beta}. 
\end{equation}
We need an estimate for the terms of form $\sum_{n=n_0}^{N-1} n^{\psi}$.  Note that for any fixed $\psi<0$, 
\[
\int_{n_0}^{N-1} n^{\psi}\leq \sum_{n=n_0}^{N-1} n^{\psi}\leq \int_{n_0-1}^{N-2} n^{\psi}, 
\]
so for some constant $D_2$, that may depend on $n_0$ and $\psi$, the following holds for
\[
\psi_N-D_2\leq \sum_{n=n_0}^{N-1} n^{\psi}\leq \psi_N+D_2, \quad \text{where}\quad \psi_N:=\begin{cases}
\frac{N^{1+\psi}}{1+\psi} &\text{if } \psi> -1,\\
\log N&\text{if }\psi=-1,\\
0&\text{if }\psi<-1.
\end{cases}
\] 
By plugging these estimates into \eqref{tmp:Ldiff}, we can also bound $\ell(m_{n_0})$. As a consequence, if $\|\nabla \ell (m_n)\|^2\geq \epsilon^2$ for all $n\leq N$, there is some constant $D_3$,
\[
0\leq -c\epsilon^2 \log N+D_3\quad\text{if}\quad \gamma=0,
\]
\[
0\leq -\frac{c\epsilon^2}{\gamma}  N^\gamma+D_3N^{\max\{0, \frac32\gamma-\frac12-\frac12\beta\}}\quad\text{if}\quad \gamma>0.
\]
This leads to our claim since they do not hold with our choice of $N$. 
\end{proof}

\begin{proof}[Proof of Theorem \ref{thm:convex}]
By  strong convexity at $m_n$, we have
\[
\ell(u^*)-\ell(m_n)\geq \langle u^*-m_n, \nabla \ell(m_n)\rangle\geq -\|\nabla \ell(m_n)\|\|u^*-m_n\|. 
\]
By strong convexity at $u^*$, we have
\[
\ell(m_n)-\ell(u^*)\geq \lambda_c \|u^*-m_n\|^2. 
\]
Therefore, we have
\[
\|\nabla \ell(m_n)\|^2\geq \frac{|\ell(m_n)-\ell(u^*)|^2}{\|u^*-m_n\|^2}\geq \lambda_c (\ell(m_n)-\ell(u^*)). 
\]
Subsitituing this into \eqref{tmp:Ldiff}, we have the following with $c=h_0 \kappa_m \sigma_m$, 
\[
( \ell(m_{n+1})-\ell(u^*))\leq (1-\lambda_ccn^{\gamma-1})( \ell(m_{n})-\ell(u^*)) + D_1 n^{\frac32\gamma-\frac32-\frac12\beta},\quad n\geq n_0.
\]
Then by Gronwalls inequality, we have that 
\[
 \ell(m_{N})-\ell(u^*)\leq S_1+S_2,\quad
 \]
 where
 \[
 S_1= ( \ell(m_{n_0})-\ell(u^*))\prod_{n=n_0}^{N-1} (1-\lambda_c cn^{\gamma-1}), \quad S_2=D_1 \sum_{j=n_0}^{N-1} j^{\frac32\gamma-\frac32-\frac12\beta} \prod_{n=j+1}^{N-1} (1-\lambda_c cn^{\gamma-1}). 
\]
To bound $S_1$, we note that $1-x\leq \exp(-x)$, so if $\gamma=0$
\begin{align*}
\prod_{n=n_0}^{N-1} (1-\lambda_c cn^{-1})\leq \exp\left(-\lambda_c c \sum_{n=n_0}^{N-1} n^{-1}\right)&\leq \exp(-\lambda_c c (\log (N-1)-\log (n_0))) \\
&=\left(\frac{n_0}{N-1}\right)^{\lambda_c c}. 
\end{align*}
When $0<\gamma<1$,
\[
\prod_{n=n_0}^{N-1} (1-\lambda_c cn^{\gamma-1})\leq \exp\left(-\lambda_c c \sum_{n=n_0}^{N-1} n^{\gamma-1}\right)\leq \exp\left(-\frac{\lambda_c c}{\gamma}( (N-1)^{\gamma}-(n_0)^{\gamma})\right).
\]
It is evident in both cases, the upper bound of $S_1$ decays to zero faster than the one in the Theorem's statement. To bound $S_2$, we will show that  when $j$ is sufficiently large, 
\[
(j+1)^{-\alpha} \prod_{n=j+1}^{N-1} (1-\lambda_c cn^{-1})- j^{-\alpha}\prod_{n=j}^{N-1} (1-\lambda_c cn^{-1})\geq j^{\frac32\gamma-\frac32-\frac12\beta} \prod_{n=j+1}^{N-1} (1-\lambda_c cn^{-1}).
\]
Since $\prod_{n=j+1}^{N-1} (1-\lambda_c cn^{-1})$ is a common factor, all we require to show is
\begin{equation}
\label{tmp:j1}
(1-\tfrac1{j+1})^{\alpha}- (1-\lambda_c c j^{\gamma-1})\geq  j^{\alpha+\frac32\gamma-\frac32-\frac12\beta}. 
\end{equation}
Note that \eqref{tmp:j1} is the identity when $\frac1j=0$, therefore we Taylor expand \eqref{tmp:j1} in terms of $\frac1j$, which is
\[
-\frac\alpha j+\lambda_c c j^{\gamma-1}\geq j^{\alpha+\frac32\gamma-\frac32-\frac12\beta}+(\text{higher order terms}).
\]
For this to hold for $j$  sufficiently large, we just need 
\[
\text{If } \gamma>0,\quad\gamma-1>\frac32\gamma+\alpha-\frac32-\frac12 \beta\Rightarrow \alpha<-\frac12\gamma+\frac12\beta+\frac12,
\]
\[
\text{If } \gamma=0,\quad\lambda_c c -\alpha>0, \quad \alpha<\frac12+\frac12\beta. 
\]
This is part of the parameter assumption of $\alpha$. Therefore, by using a larger $n_0$, we find
\begin{align*}
\sum_{j=n_0}^{N-1} j^{\frac32\gamma-\frac32-\frac12\beta} \prod_{n=j+1}^{N-1} (1-\lambda_c cn^{-1})
&\leq  \sum_{j=n_0}^{N-1}( (j+1)^{-\alpha} \prod_{n=j+1}^{N-1} (1-\lambda_c cn^{-1}) \\ &- j^{-\alpha}\prod_{n=j}^{N-1} (1-\lambda_c cn^{-1}))\\
&\leq N^{-\alpha}.
\end{align*}
This concludes our proof. 
\end{proof}

\section{Iterative Scheme}

\begin{algorithm}[H]
\caption{\textcolor{black}{Ensemble Kalman Inversion - (ETKF formulation \cite{BEM01})}}
\label{alg:1}
\begin{algorithmic}[1]
\State \textcolor{black}{Draw an initial ensemble $\{u^{(i)}_0\}_{i=1}^{K} \sim \mathcal{N}(0,C_0)$ with $K$ ensemble members.} 
\State  \textcolor{black}{Choose $0 < \gamma < 1$  where $ \gamma-1\leq\beta\leq \gamma$}.
\State  \textcolor{black}{Repeat steps 4-8 for $n=0,\ldots,N-1$,}
\\  \textcolor{black}{Compute cross covariances using \eqref{eq:samp_cov}.}
\State \textcolor{black}{Update mean through}
$$
\textcolor{black}{m_{n+1}=m_n+ C^{up}_n( C^{pp}_n +\Gamma_+ )^{-1}(z -\calH(m_n)), \quad {m}_n = \frac{1}{K}\sum^{K}_{i=1}u^{(i)}_n.}
$$
\State \textcolor{black}{Define ensemble spread $S_n \in \mathbb{R}^{d_u \times K}$ as}
$$
\textcolor{black}{S_n = [u^{(1)}_n - {m}_n, \ldots, u^{(K)}_n - {m}_n].}
$$
\State \textcolor{black}{Find a matrix $T_n \in \mathbb{R}^{d_u \times d_u}$, such that}
$$
\textcolor{black}{\frac{1}{K-1}T_n  S_n \otimes T_n  S_n = C^{uu}_{n+1}=C^{uu}_n- C^{up}_n (C^{pp}_n +h^{-1}_n \Gamma_+)^{-1}C_n^{pu}+\alpha^2_n \Sigma,}
$$
\textcolor{black}{where $$\alpha_n^2=\alpha_0^2h_0^{-1} n^{2\gamma-\beta-2}.$$}
\textcolor{black}{One possibility is letting $T_n=(C^{uu}_{n+1})^{\frac{1}{2}}(C^{uu}_{n})^{-\frac12}$.}
\State \textcolor{black}{The spread is updated to $\tilde{S}_n = T_n S_n $, and ensemble members updated as}
$$
\textcolor{black}{u^{(i)}_{n+1} = m_n + \tilde{S}^{(i)}_n.}
$$
%
\end{algorithmic}
\end{algorithm}



\end{document}